\documentclass[a4paper,11pt, reqno]{amsart}
\usepackage[T1]{fontenc}
\usepackage[english]{babel}

\usepackage{amsmath}
\usepackage{amssymb}
\usepackage{amsthm}
\usepackage[foot]{amsaddr}
\usepackage{bbm}
\usepackage{mathrsfs}
\usepackage{epsfig}
\usepackage{enumerate}
\usepackage{mathtools}
\usepackage{graphicx}
\usepackage[usenames,dvipsnames]{xcolor}
\usepackage[colorlinks=true,linkcolor=NavyBlue,urlcolor=RoyalBlue,citecolor=PineGreen,bookmarks=true,bookmarksopen=true,bookmarksopenlevel=2,unicode=true,linktocpage]
{hyperref}
\usepackage[margin=1in]{geometry}

\newtheorem*{lemma*}{Lemma}
\newtheorem{theoremm}{Theorem}
\newtheorem{theorem}{Theorem}[section]
\newtheorem{lemma}[theorem]{Lemma}
\newtheorem{sublemma}[theorem]{Sublemma}
\newtheorem{proposition}[theorem]{Proposition}
\newtheorem{corollary}[theorem]{Corollary}

\newtheorem{remark}[theorem]{Remark}

\def\d{\,{\rm d}}
\def\R{{\mathbb R}}

\def\D{{\mathcal D}}

\DeclareMathOperator{\Var}{Var}

\def\diams{\raisebox{1.5pt}{$\scriptstyle\diamond$} }

\title{Winding and intersection of Brownian motions}
\author{Isao Sauzedde}
\address{University of Oxford}
\email{isao.sauzedde@maths.ox.ac.uk}


\begin{document}
\begin{abstract}
We study the set of points $\mathcal{D}_{n,m}$ around which two independent Brownian motions wind at least $n$ (resp. $m$) times. We prove that its area is asymptotically equivalent, in $L^p$ and almost surely, to $\frac{\ell(\R^2)}{4\pi^2 n m}$, where $\ell$ is the intersection measure of the two trajectories.
We also prove that the properly scaled Lebesgue measure carried by $\mathcal{D}_{n,m}$ converges almost surely weakly toward $\ell$.
\end{abstract}
\maketitle

\tableofcontents

\section{Introduction}

In \cite{Werner}, motivated by a question of J.-F. Le Gall about a possible Green's formula for the planar Brownian motion and with applications in physics \cite{Comtet,Edwards}, W.Werner studied the Lebesgue measures $D^X_n$ and $A^X_n$ of the sets
\[\mathcal{D}^X_n=\{z\in \R^2: \theta_X(z)\geq n\}, \quad \mathcal{A}^X_n=\{z\in \R^2: \theta_X(z)= n\}\]
of the points around which the planar Brownian motion $X:[0,1]\to\R^2$ winds at least $n$ times (resp. exactly $n$ times).
By a very careful asymptotic analysis of the joint law of winding around two points, he managed to show that $D^X_n$ (resp. $A^X_n$)
 is asymptotically equivalent to~$\frac{1}{2\pi n}$ (resp.~$\frac{1}{2\pi n^2}$) in $L^2$.

In \cite{LAWA}, motivated by the same questions, we have pushed the asymptotic expansion of $D_n^X$ by showing that, both in $L^{\infty -}$ and in the almost sure sense, \[D^X_n=\frac{1}{2\pi n}+O(n^{-\frac{3}{2}+\epsilon} ).\]
It allowed us to proved an almost sure version of the stochastic Green's formula,
\[ \int X^1 \d X^2\overset{a.s.}= \lim_{n\to \infty} \sum_{k=-n}^{n} k A^X_k.
\]

The present work was initially motivated by the idea of pushing the asymptotic expansion further. We think that such an expansion should lead for example to a better understanding of magnetic impurities as described in \cite{Desbois} and \cite[Section 6]{Comtet2}, and also of the winding field associated with the Brownian loop soup \cite{Camia,Camia2}. In fact, a better understanding of $\D^X_n$ should allow to relates this winding field with the multiplicative chaos of the Brownian loop soup \cite{Lupu}.
We also expect that this asymptotic expansion should allow to define stochastic integrals of extremely irregular random $1$-forms, extending the framework we developed in \cite{DSI}.

In fact, the second order term in the asymptotic expansion of $D^X_n$ seems to be deeply related to the self-intersections of the Brownian motion, and actually determined by its self-intersection local time. In order to prove such a result, it is needed to first understand the intersection of the large winding sets of different Brownian pieces.

This paper is devoted to the study of the large $n,m$ asymptotic of the joint large winding set
\[ \mathcal{D}^{X,Y}_{n,m}=\{z\in \mathbb{R}^2: \theta_X(z)\geq n, \theta_Y(z)\geq m\},\]
where $X$ and $Y$ are two planar Brownian motions from $[0,1]$ to $\R^2$.

Our main results are the following. Let $\ell^{X,Y}$ be the intersection measure of $X$ and $Y$, as defined for example in \cite[(1-a)]{LeGall} or in \cite{Geman}.
For a curve $X$ in the plane, let $\bar{X}$ be the curve obtained by closing $X$ with a straight line segment between its endpoints. For $z\in \mathbb{R}^2$ outside the range of $\bar{X}$, let then $\theta_X(z)\in \mathbb{Z}$ be the winding number of $\bar{X}$ around $z$.
Let $D^{X,Y}_{n,m}=|\mathcal{D}^{X,Y}_{n,m}|$ be the Lebesgue measure of $\mathcal{D}^{X,Y}_{n,m}$, and let $\mu^{X,Y}_{n,m}$ be the measure given by
\[ \frac{\d \mu^{X,Y}_{n,m}}{\d z}(z)=nm \mathbbm{1}_{\mathcal{D}^{X,Y}_{n,m}}(z).\]
\begin{theoremm}
\label{th:intersection}
  Let $m$ be a non-decreasing function of $n$, and assume that there exists $0<c_1<c_2\leq 1 $ such that $n^{c_1}<m\leq n^{c_2}$ for all positive integers $n$. Then,
  \begin{itemize}
  \item
  Both in $L^p$ for all $p\in [1,\infty)$ and in the almost sure sense,
  \[nm D^{X,Y}_{n,m}\underset{n\to \infty}{{\longrightarrow}} \frac{\ell^{X,Y}(\mathbb{R}^2)}{4\pi^2 }.\]
  \item For all $\epsilon>0$,
  \[
  m^{\frac{1}{2}-\epsilon} \Big( nm D^{X,Y}_{n,m}-  \frac{\ell^{X,Y}(\mathbb{R}^2)}{4\pi^2 } \Big)\underset{n\to \infty}{\overset{L^p}{\longrightarrow}} 0.
  \]
  \item
  For all $\epsilon>0$,
  \[
   m^{\frac{1}{2}\frac{c_1(1+c_2)}{c_2(1+c_1)}-\epsilon } \Big( nm D^{X,Y}_{n,m}-  \frac{\ell^{X,Y}(\mathbb{R}^2)}{4\pi^2 } \Big)\underset{n\to \infty}{\overset{a.s.}{\longrightarrow}} 0.
  \]
  \end{itemize}
\end{theoremm}
We endow the space $\mathcal{M}$ of finite measures over $\R^2$ with the $1$-Wasserstein distance
\[ d_1(\mu,\nu)=\sup\{ \int_{\R^2} f \d (\mu-\nu): f \mbox{ $1$-Lipschitz}, f(0)=0 \},\]
which metrizes the weak convergence of measures supported on a given compact.
\begin{theoremm}
  \label{th:mesure}
  Let $m$ be a non-decreasing function of $n$, and assume that there exists $\epsilon>0$ such that $n^\epsilon<m\leq n$ for all positive integer $n$.

  Then, for all $p\in [1,\infty)$, the measure $\mu^{X,Y}_{n,m}$ converges in $L^p(\Omega, (\mathcal{M},d_1) )$ toward $\frac{\ell^{X,Y}}{4\pi^2 }$.

  Besides, almost surely,
  \[ \mu^{X,Y}_{n,m} \implies \frac{\ell^{X,Y}}{4\pi^2}.\]
\end{theoremm}
To be completely clear, the first statement means that for all $p\in[1,\infty)$
\[
\mathbb{E}\Big[ \Big|\sup\{
\int_{\R^2} f \d (\mu^{X,Y}_{n,m}-\ell^{X,Y} ): f \mbox{ $1$-Lipschitz}, f(0)=0 \}  \Big|^p\Big]\underset{n\to \infty}{\longrightarrow} 0.
\]

We expect our method to extend without additional difficulty to the study of the joint large winding sets of three or more Brownian motions. Let us remark that the general pattern is very similar to the one we used in \cite{LAWA} for a single Brownian motion, which shows its robustness, even though the technical details are much more subtle here.

\vspace{0.2cm}
In \cite{BWE}, we considered already the normalized Lebesgue measure carried by $\D^X_n$, with the intent of extending the stochastic Green's formula to general $1$-forms.\footnote{Part of this program was achieved in the author's PhD manuscript, Theorem 4.6.1. The regularisation procedure is shown to converge under mild regularity conditions on the $1$-form, but the limit is not identified as a Stratonovich integral.} The present work answers almost entirely the conjecture we made in that paper.

To complete this introduction, let us mention that the study of the Brownian windings started with Spitzer who gave the large time asymptotics of $\theta_X(z)$ \cite{Spitzer}, continued by Yor who gave an explicit expression for the law of $\theta_X(z)$ \cite{Yor} (see also \cite{Mansuy}), and followed by many mathematicians in various settings \cite[...]{Bertoin,Franchi,LeGall2,LeGall3,LeGall4,Pitman,Pitman2,Zhan}.

\section{Notations and  general ideas}
In the following, it is always assumed that $m$ is a non-decreasing integer-valued function of the integer $n$, and that there exists $c>0$ such that $n^c<m\leq n$ for all positive integer $n$. It is also assumed that $T$ is another integer-valued function of $n$, and that $m$ and $T$ are larger than $2$.

Unless otherwise specified, all the Brownian motions are defined from $[0,1]$ to $\R^2$. Under $\mathbb{P}_{x,y}$, $X$ and $Y$ are two independent Brownian motions starting respectively from $x$ and $y$. When it is not necessary to specify these starting points, we simply write $\mathbb{P}=\mathbb{P}_{x,y}$. When we write $\mathbb{E}_{T^{\frac{1}{2}}X_{iT^{-1}}, T^{\frac{1}{2}} Y_{jT^{-1}} }[f(\hat{X},\hat{Y})]$,
it should be understood that $\hat{X},\hat{Y}$ are two independent Brownian motions, defined from $[0,1]$ to $\R^2$, starting respectively from $T^{\frac{1}{2}}X_{iT^{-1}}$ and $T^{\frac{1}{2}} Y_{jT^{-1}}$, and such that
$(\hat{X}-T^{\frac{1}{2}} X_{iT^{-1}},\hat{Y}-T^{\frac{1}{2}} Y_{jT^{-1}})$ is independent from $(X_{iT^{-1}},Y_{jT^{-1}})$.



For $i\in \{1,\dots, T\}$, we define $X^i$ (resp. $Y^i$) as the restriction of $X$ (resp. $Y$) to the interval $[(i-1)T^{-1},iT^{-1}]$.
We denote by $\theta_X^i(z)$ the integer winding of $X^i$ around $z$.
For two generic curves $\tilde{X}$ and $\tilde{Y}$, we define the sets
\[ 
\mathcal{D}^{\tilde{X}}_n=\{z\in \mathbb{R}^2: \theta_{\tilde{X}}(z)\geq n\},\quad
\mathcal{D}^{{\tilde{X}},{\tilde{Y}}}_{n,m}=\{z\in \mathbb{R}^2: \theta_{\tilde{X}}(z)\geq n, \theta_{\tilde{Y}}(z)\geq m\}.\]
We replace the superscripts $X^i$ and $X^i,Y^j$ with the superscripts $i$ and $i,j$.

For each of these sets, we replace the curly letter with a straight one to designate its area: for example, $D^i_n=|\mathcal{D}^i_n|=|\mathcal{D}^{X^i}_n|$.

The $2$-dimensional heat kernel is denoted $p_t(x,y)=(2\pi t)^{-1}e^{-\frac{|y-x|^2}{2t}}$, and we write $P_tf$ for $p_t(0,\cdot)*f$.\\

In Section \ref{sec:error}, we will first show that the quantity
\[
\Sigma_{n,m,T}=nm \sum_{i,j=1}^T D^{i,j}_{n,m}
\] is a good approximation of $nmD_{n,m}$. In Section \ref{sec:l2}, we will give an asymptotic estimation of $\Sigma_{n,m,T}$, proving therefore the $L^2$ convergence of $nm D^{X,Y}_{n,m}$. In Section \ref{sec:boot}, we improve the convergence rate. In Section \ref{sec:lp}, we extend the convergence to the $L^p$ and almost sure sense. Finally, Section \ref{sec:mesure} is devoted to the convergences for the measure
$\mu^{X,Y}_{n,m}$.\\

The idea behind our method, that we used previously in \cite{LAWA} and \cite{BWE}, is that the winding $\theta_X(z)$ is equal to the sum of all the windings $\theta^i_X$, plus a piecewise-linear part. When $\theta_X(z)$ is large, it is in general only one of these pieces $X^i$ which have a large winding, so that $\theta_X(z)$ is then roughly equal to $\theta_X^i$. It follows that the set $\mathcal{D}^X_N$ is roughly equal to the union of the sets $\mathcal{D}^i_N$, and that these sets are roughly disjoint, so that a kind of central limit theorem occurs. 

Following again the ideas introduced in \cite{LAWA}, we first presume that neither the rate at which we can show $\Sigma_{n,m,T}$ to converge, nor the choice of $T$ we take, are actually relevant, as soon as the convergence rate is some power of $T$ and $T$ is some power of $m$. A rather simple procedure ultimately allows to drastically improve this convergent rate. 


\section{Comparison between $D_{n,m}$ and $\Sigma_{n,m,T}$}
\label{sec:error}
\subsection{$L^2$ bounds}
The goal of this section is to estimate the difference between $D^{X,Y}_{n,m}$ and $\sum_{i,j=1}^T D^{i,j}_{n,m}$. The method that we use is very similar to the one we used in \cite{LAWA}, but we have drastically simplify some technical steps, in order to deal with a more general situation without going into tremendous computations.
 The cost of these simplification is a slight depreciation of the result which does not spoil its interest. 

We consider here a family of $d$ independent planar Brownian motions $X_1,\dots, X_d$, starting from deterministic points $x_1,\dots, x_d$. 
We consider a family
$\mathbf{n}_j=(n_{j,1},\dots, n_{j,k_j})$ of integers, an integer $T$ greater than $2$ and such that $\frac{\log (n_{j,i})}{\log (T)}\in[c,c^{-1}]$ for all possible choice of indices $i,j$. We look at the limit when $T$ and the $n_{j,i}$ go to infinity with this condition fullfilled.
We also set, for each $j\in \{1,\dots, d\}$ a collection $\mathbf{i}_j=(i_{j,1},\dots i_{j,k_j })\in \{1,\dots, T\}^{k_j}$ with $i_{j,l}\neq i_{j,l'}$ for $l\neq l'$. We then define
\[
\mathcal{R}^{\mathbf{i}_1,\dots, \mathbf{i}_d}_{\mathbf{n}_1,\dots, \mathbf{n}_d}=\{ z\in \mathbb{\R}^2: \forall j\in \{1,\dots, d\}, \forall l\in \{1,\dots, k_j\}, |\theta^{i_{d,l}}_{X_j}(z)| \geq n_{d,l} \}.
\]
Be careful about the absolute values. We also set $R^{\mathbf{i}_1,\dots, \mathbf{i}_d}_{\mathbf{n}_1,\dots, \mathbf{n}_d}=|\mathcal{R}^{\mathbf{i}_1,\dots, \mathbf{i}_d}_{\mathbf{n}_1,\dots, \mathbf{n}_d}|$.

\begin{proposition}
  \label{prop:l2bound}
  Let $cn>0$, and let $p,d,k_1,\dots, k_d$ be integers. Then, there exists a constant $C$ and a real $q$ such that for all possible choice of integer $T\geq 2$ and families $\mathbf{n}_1, \dots , \mathbf{n}_d$ and $\mathbf{i}_1, \dots , \mathbf{i}_d$ with $\mathbf{n}_j=(n_{j,1},\dots, n_{j,k_j})$, $\mathbf{i}_j=(i_{j,1},\dots, i_{j,k_j})$,
  and such that $\frac{\log(n_i)}{\log(T)}\in [c,c^{-1}]$,
  \[
  \sup_{x_1,\dots x_d \in \R^2}
  \mathbb{E}_{x_1,\dots, x_d}\big[ \big(R^{\mathbf{i}_1,\dots, \mathbf{i}_d}_{\mathbf{n}_1,\dots, \mathbf{n}_d}\big)^p]^{\frac{1}{p}}\leq C \log(T)^{q} T^{-\frac{2}{p}} \prod_{j=1}^d \prod_{l=1}^{k_j} n_{j,l}^{-1}.
  \]
\end{proposition}
\begin{proof}
  We arrange the families to have $i_{j,1}<\dots <i_{j,k_j}$.
  For $\mathfrak{z}=(z_1,\dots, z_p)\in (\R^2)^p$, and $n\in \mathbb{N}$, let $f_{n}(\mathfrak{z})= \mathbb{P}_0(\forall j\in \{1,\dots, p\}, \theta(z_j)\geq n)$.
  Let also $P_t f (\mathfrak{z})=\int_{\R^2} p_t(0,y) f(\mathfrak{z}-y) \d y$, where $\mathfrak{z}-y=(z_1-y,\dots,z_p-y )$.

  Then, \[\mathbb{P}_x(\theta^1(z)\geq n)= \mathbb{P}_0(\theta(T^{\frac{1}{2}}(z-x))\geq n )=f_n(T^{\frac{1}{2}}(z-x)  ),\]
  \[
  \mathbb{P}_x(\theta^i(z)\geq n)=P_{i-1}f_n(T^{\frac{1}{2}}(z-x) ),\]
  and more generally
  \[
  \mathbb{P}_x(\forall j\in\{1,\dots, p\},  \theta^i(z_j)\geq n )=
  P_{i-1}f_{n}(T^{\frac{1}{2}}(\mathfrak{z}-x) ).
  \]
  Let $k=\sum_{j=1}^d k_j$.
  Then,
  \begin{align}
  \mathbb{E}_{x_1,\dots, x_d}&[ \big(R^{\mathbf{i}_1,\dots, \mathbf{i}_d}_{\mathbf{n}_1,\dots, \mathbf{n}_d}\big)^p]
  = \int_{(\R^2)^p} \prod_{j=1}^d \mathbb{P}_{x_j}\big(\forall l \in \{1,\dots, k_j\}, \forall q\in \{1,\dots,p\}, \    \theta_{X_j}^{i_{j,l}}(z_q)\geq n_{j,l}\big) \d \mathfrak{z}\nonumber\\
  &=\int_{(\R^2)^p} 
  \prod_{j=1}^d \prod_{l=1}^{k_j}
  \mathbb{P}_{x_j}( \forall q\in \{1,\dots,p\}, \    \theta^{i_{j,l}}(z_q)\geq n_{j,l}     \big|\forall l'<l,  \    \theta^{i_{j,l'}}(z_q)\geq n_{j,l'}       \big)  \d \mathfrak{z}\nonumber\\
  &\leq \int_{(\R^2)^p}
  \prod_{j=1}^d \prod_{l=1}^{k_j}
  \sup_{a\in \R^2} \mathbb{P}_{x_j}( \forall q\in \{1,\dots,p\}, \    \theta^{i_{j,l}}(z_q)\geq n_{j,l}     \big|X_{1,i_{j,l-1}T^{-1}}=a \big)  \d \mathfrak{z}\nonumber\\
  &= \int_{(\R^2)^p}\prod_{j=1}^d \prod_{l=1}^{k_j} \sup_{a\in \R^2} P_{i_{j,l}-i_{j,l-1}-1 }f_{n_{j,l}}(\sqrt{T} \mathfrak{z}-a ) \d \mathfrak{z}\nonumber\\
  &\leq \prod_{j=1}^d \prod_{l=1}^{k_j} \Big(\int_{(\R^2)^p} \sup_{a\in \R^2} \big(P_{i_{j,l}-i_{j,l-1}-1 }f_{n_{j,l}}(\sqrt{T} \mathfrak{z}-a )\big)^k \d \mathfrak{z} \Big)^{\frac{1}{k}}
  \label{eq:bad1}\\
  &=\prod_{j=1}^d \prod_{l=1}^{k_j} \Big(T^{-2p} \int_{(\R^2)^p} \big(P_{i_{j,l}-i_{j,l-1}-1 }f_{n_{j,l}}( \mathfrak{z}' )\big)^k \d \mathfrak{z}' \Big)^{\frac{1}{k}}\nonumber\\
  &=T^{-2}\prod_{j=1}^d \prod_{l=1}^{k_j} \big\|P_{i_{j,l}-i_{j,l-1}-1 }f_{n_{j,l}} \big\|_{L^k} 
  \leq T^{-2}\prod_{j=1}^d \prod_{l=1}^{k_j} \big\|f_{n_{j,l}} \big\|_{L^k}.\label{eq:bad2}
  \end{align}
  We fix a positive real number $\beta$, and we set \[E_\beta=\{\mathfrak{z}=(z_1,\dots, z_p) \in (\R^2)^p :  \min \{|z_i|,|z_i-z_j|: i\neq j \}\leq T^{-\beta}.\] We have shown in \cite[Sublemma 2.2]{LAWA} that $f$ admits the following bounds, for some constant $C$ that depends only on $\beta$ and $c$:
  \begin{itemize}
  \item
  $f_n(\mathfrak{z} )\leq C\log(n)^pn^{-p}$ for $\mathfrak{z} \notin E_\beta$,
  \item $f_n(\mathfrak{z})\leq 4 \exp\Big( - \frac{\max\{ |z_i|^2:i\in \{1,\dots, p\} }{4} \Big)$.
  \end{itemize}
  By decomposing $(\R^2)^p$ into the disjoint union of $E_\beta \cap B(0,\log(T))^p$, $B(0,\log(T))^p\setminus E_\beta$, and $(\R^2)^p\setminus B(0,\log(T))^p$, we deduce that
  \[ \|f_n\|_{L^k}\leq C (\log(T)^{\frac{2p-2}{q}}T^{-\frac{2\beta}{q}} + \log(T)^2 \log(n)^p n^{-p}+ T^{-\frac{\log(T)}{4}}).\]
  For $\beta$ sufficiently large, and with $n$ ranging over the $n_{i,j}$ so that $\frac{\log(n)}{\log(T)}$ is bounded, this reduces to $\|f_n\|_{L^k}\leq log(T)^{2p+2} n^{-p}$, and we end up with
  \[
  \mathbb{E}_{x_1,\dots, x_d}[ \big(R^{\mathbf{i}_1,\dots, \mathbf{i}_d}_{\mathbf{n}_1,\dots, \mathbf{n}_d}\big)^p]\leq \log(T)^{2kp+2k} T^{-2}\prod_{j=1}^d \prod_{l=1}^{k_j} n_{j,l}^{-p},\]
  as announced.
\end{proof}
Remark that we were loose at two places, on lines \eqref{eq:bad1} and \eqref{eq:bad2}. By being more subtle, we should normally have extra factors $(i_{j,l}-i_{j,l-1}+1)^{-1}$, but the proof should be much longer to get them.

We also present a similar bound in which we are not considering small pieces, but the whole trajectories. We will use this bound a lot in the last sections.
For $X_1,\dots, X_d$ independent planar Brownian motions starting from $x_1,\dots, x_d$, and positive integers $n_1,\dots,n_d$, let
\[\mathcal{R}^{X_1,\dots,X_d}_{n_1,\dots, n_d}=\{ z\in \R^2: \forall j\in \{1,\dots, d\}, |\theta_{X_j}(z)|\geq n_i\},\]
and $R^{X_1,\dots,X_d}_{n_1,\dots, n_d}=|\mathcal{R}^{X_1,\dots,X_d}_{n_1,\dots, n_d}|$.
\begin{lemma}
  \label{le:globalBound}
  For all positive integer $p$, there exists $c,C$ such that for all integers $n_1,\dots, n_d\geq 2$ and $x_1,\dots x_d\in \R^2$,
  \[
  \mathbb{E}_{x_1,\dots, x_d}[(R^{X_1,\dots,X_d}_{n_1,\dots, n_d})^p]^{\frac{1}{p}}\leq C \log(n_1\dots n_d)^c n_1^{-1}\dots n_d^{-1}.
  \]
\end{lemma}
\begin{proof}
  This is similar to but simpler than the previous proof. We have
  \begin{align}
  \mathbb{E}_{x_1,\dots, x_d}[(R^{X_1,\dots,X_d}_{n_1,\dots, n_d})^p]
&  = \int_{(\R^2)^p} \prod_{j=1}^d \mathbb{P}_{x_j}\big( \forall q\in \{1,\dots,p\}, \    \theta_{X_j}(z_q)\geq n_{j}\big) \d z_1\dots \d z_d\nonumber\\
  &=\int_{(\R^2)^p}
  \prod_{j=1}^d f_{n_j}(z_1,\dots, z_d)\d z_1\dots \d z_d\\
  &\leq \prod_{j=1}^d  \Big(\int_{(\R^2)^p}
  f_{n_j}(z_1,\dots, z_d)^d \d z_1\dots \d z_d\Big)^{\frac{1}{d}}.
  \end{align}
  Decomposing $(\R^2)^p$ as in the previous proof, with $T=n_1\dots n_d$, and with $\beta$ sufficiently small, we get $\|f_{n_j}\|_{L^d}\leq \log(T)^{p+2} n_j^{-p}$, from which the lemma follows.
\end{proof}

We now summon some inclusions from \cite[Equations (24) and (25)]{LAWA}. We invite our reader to understand these inclusions by themself rather than by looking the formal proof in \cite{LAWA}, which is not very enlightening.\footnote{Remember that an additional factor $T$ come from the piecewise-linear part of the path, $\theta^{pl}=\theta_X-\sum_i \theta^i$.} The meaning of these inclusions is simply that for a sum to be large, some of the summands must be large.
\begin{lemma}
  Let $n,p,T$ be such that $\frac{n}{3}>T(p+1)$. Then,
  \begin{align*}
  \sum_{i=1}^T \mathcal{D}^i_{n+T(p+1)}\setminus \bigcup_{\substack{i,j=1\\i\neq j}}^T  R^{\{i,j\}}_{\{\frac{n}{3},p\} }
  \subseteq\mathcal{D}^X_n\subseteq
  \sum_{i=1}^T \mathcal{D}^i_{n-T(p+1)}\cup \bigcup_{\substack{i,j=1\\i\neq j}}^T  R^{\{i,j\}}_{\{\frac{n}{3},p\} } \cup  \bigcup_{\substack{i,j,k=1\\i\neq j\neq k\neq i}}^T  R^{\{i,j,k\}}_{\{p,p,p\} }.
\end{align*}
\end{lemma}
Using this decomposition on both $\mathcal{D}^X_n$ and $\mathcal{D}^Y_m$, and using the inclusion-exclusion principle, we deduce
\begin{corollary}
  \label{coro:encadrement}
  Let $n,m,p,q,T$ be such that $\frac{n}{3}>T(p+1)$ and $\frac{m}{3}>T(q+1)$. Then,
  \[
  \sum_{i,j=1}^T D^{i,j}_{n+T(p+1), m+T(q+1)} - R_1
  \leq
  D^X_{n,m}\leq
  \sum_{i,j=1}^T D^{i,j}_{n-T(p+1), m-T(q+1)}+ R_2
  \]
  where
  \begin{align*}
  R_1&
  =
  \sum_{\substack{i,j,k\\i\neq j }}  R^{\{i,j\},\{k\}}_{\{\frac{n}{3},p\},\{  m+T(q+1) \} }
  + \sum_{\substack{i,j,k\\j\neq k }}  R^{\{i\},\{j,k\}}_{\{  n+T(p+1) \},\{\frac{m}{3},q\}, }
  +\sum_{\substack{i,j,k,l\\ i\neq j, \\ k\neq l}}  R^{\{i,j\},\{k,l\}}_{\{\frac{n}{3},p\},\{\frac{m}{3},q\} }\\
\shortintertext{and}
  R_2&
  =
  \ \ \sum_{\substack{i,j,k\\i\neq j }}\ \   R^{\{i,j\},\{k\}}_{\{\frac{n}{3},p\},\{  m-T(q+1) \} }
  \ + \ \sum_{\substack{i,j,k\\j\neq k }}  R^{\{i\},\{j,k\}}_{\{  n-T(p+1) \},\{\frac{m}{3},q\}, }
  +\sum_{\substack{i,j,k,l\\ i\neq j\neq k\neq i}}  R^{\{i,j,k\},\{l\}}_{\{p,p,p\},\{  m-T(q+1) \} }\\&
  +\sum_{\substack{i,j,k,l\\ j\neq k\neq l\neq j}}  R^{\{i\},\{j,k,l\}}_{\{n-T(p+1)\},\{ q,q,q\} }
  +\sum_{\substack{i,j,k,l\\ i\neq j, \\ k\neq l}}  R^{\{i,j\},\{k,l\}}_{\{\frac{n}{3},p\},\{\frac{m}{3},q\} }
  +\sum_{\substack{i,j,k,l,r\\ i\neq j,\\ k\neq l\neq r\neq k}}  R^{\{i,j\},\{k,l,r\}}_{\{\frac{n}{3},p\},\{q,q,q\} }\\&
  +\sum_{\substack{i,j,k,l,r\\ i\neq j\neq k\neq i, \\ l\neq r}}  R^{\{i,j,k\},\{l,r\}}_{\{p,p,p\},\{\frac{m}{3},q\} }
  +\sum_{\substack{i,j,k,l,r,s\\ i\neq j\neq k\neq i, \\ l\neq r\neq s\neq l}}  R^{\{i,j,k\},\{l,r,s\}}_{\{p,p,p\},\{q,q,q\} }.
  \end{align*}
\end{corollary}
From Proposition \ref{prop:l2bound}, we obtain that for all $\epsilon>0$ and $r>1$, there exists $C,c$ such that for all $n,m,p,q,T$ with $v\geq \sqrt{T}$ and $v^{\epsilon}<T$ for $v\in \{n,m,p,q\}$, and with $Tp<n$ and $Tq<m$,
\[\mathbb{E}[R_1^r ]^{\frac{1}{r} }\leq C\log(T)^c T^{-\frac{2}{r}} \big( T^3n^{-1}m^{-1}p^{-1}+T^3n^{-1}m^{-1}q^{-1} \big),\]
and
\[\mathbb{E}[R_2^r ]^{\frac{1}{r} }\leq C\log(T)^c T^{-\frac{2}{r}} \big( T^3n^{-1}m^{-1}p^{-1}+T^3n^{-1}m^{-1}q^{-1}+ T^4p^{-3} m^{-1}+T^4n^{-1}q^{-3}\big).\]

In particular, taking $p=\sqrt{n}$ and $q=\sqrt{m}$, we have
\begin{proposition}
\label{prop:encadrement}
  For all $r>1$, there exists $C,c$ such that for all $m\leq n$,
  \[\sum_{i,j=1}^T D^{i,j}_{n+T(\sqrt{n}+1), m+T(\sqrt{m}+1)} - R\leq D^{X,Y}_{n,m}\leq\sum_{i,j=1}^T D^{i,j}_{n-T(\sqrt{n}+1), m-T(\sqrt{m}+1)}+ R,\]
  with
  \[\mathbb{E}[R^r]^{\frac{1}{r}} \leq  C\log(T)^c T^{3 -\frac{2}{r}} n^{-1}m^{-\frac{3}{2}}.\]
  Alternatively,
  \[D^{X,Y}_{n+T(\sqrt{n}+1),m+T(\sqrt{m}+1)}- R\leq\sum_{i,j=1}^T D^{i,j}_{n, m}\leq D^{X,Y}_{n-T(\sqrt{n}+1),m-T(\sqrt{m}+1)}+ R.\]
\end{proposition}

\section{Asymptotic in $L^2$ for the sum }
\label{sec:l2}
Our goal in this section is roughly to show that $\Sigma_{n,m,T}$ is equal to $\frac{\ell^{X,Y}(\R^2)}{4\pi^2}$, plus an error term which has a small $L^2$-norm when $n,m,T$ are large. Let us remark that in this section, we derive the convergence for $D^{X,Y}_{n,m}$, but we do not get the convergence rate announced in the theorem: this will be done in Section \ref{sec:boot}. In fact, apart from Subsection \ref{subsub}, the content of this section is \emph{not} necessary to obtain the convergence with the good convergence rate. The reader could skip directly to Subsection \ref{subsub} and then to Section
\ref{sec:boot}, but we think that the strategy in this section gives a different light on the subject than the proofs in Section
\ref{sec:boot}. In fact, it is even difficult to understand the precise role played by  $\ell^{X,Y}(\R^2)$ by the reading of Section \ref{sec:boot} alone.
Besides, some of the results given here concern properties of the intersection measure that we haven't found elsewhere and they might be of independent interest.

Our strategy consist on giving meaning to the following successive estimations, where $L:\R^2 \to \R_+$ is a function that will be determined later.
\begin{align*}
nm \sum_{i,j=1}^T D^{i,j}_{n,m}& \simeq nm \sum_{i,j=1}^T  \mathbb{E}[D^{i,j}_{n,m}|X_{(i-1)T^{-1}},Y_{(j-1)T^{-1}}] & \mbox{(Lemma \ref{le:tech1})}\\
&\simeq \sum_{i,j=1}^T T^{-1} L_{\sqrt{T} ( Y_{(j-1)T^{-1}}- X_{(i-1)T^{-1}} )} & \mbox{(Lemma \ref{le:tech2})}\\
&\simeq \int_0^1 \int_0^1 T L_{\sqrt{T}(Y_t-X_s)} \d s\d t &\mbox{(Corollary \ref{coro:tech3})}\\
&\simeq \ell^{X,Y}(\R^2)\ \int_{\R^2} L_z \d z. & \mbox{(Lemma \ref{le:tech4})}
\end{align*}
Remark that already after the second step, we have eliminated the dependency in $n$ and $m$, and the winding do not appear anymore after that point.

\subsection{Limitation of the noise}
The goal of this subsection is to prove the first in the series of approximations presented above.
\begin{lemma}
\label{le:tech1}
  For all $c,\epsilon>0$, there exists a constant $C$ such that for all $n,m,T$ with $T<m\leq n<T^{c^{-1}}$, for all $x,y\in \R^2$,
  \[\mathbb{E}_{x,y}\Big[\Big(T^{\frac{1}{2}-\epsilon} nm \sum_{i,j=1}^T (D^{i,j}_{n,m}-\mathbb{E}[D^{i,j}_{n,m}|X_{(i-1)T^{-1}},Y_{(j-1)T^{-1}}]) \Big)^2\Big]\leq C.\]
\end{lemma}
\begin{remark}
  {\normalfont Many of the proofs we will present have the same structure, which therefore deserve a general explanation. We expend the square into a sum over four indices $i,j,k,l$, and we split this sum in four parts: \[ \{i=k,j=l\}, \ \{i=k \mbox{ or } j=l \}, \{ i<k,j<l \mbox{ or } i>k, j<l \}, \]
  \[\{ i<k,j>l \mbox{ or } i>k, j>l \}.\]

  It will be useful to keep track of how many powers of $T$ we should be able to save, and how many of them we need to save. For example, if we consider the sum with all the four indices $i= k,j\neq l$, we need save $3$ powers of $T$ so that the sum does not diverge, plus some extra power so that it actually goes to zero sufficiently fast. We will save a factor $T^2$ from scaling. One should come from the fact that $D^{i,j}_{n,m}$ has an extremely large probability to be $0$, unless $|X_{(i-1)T^{-1}}-Y_{(j-1)T^{-1}}|$ is of the order of $T^{-\frac{1}{2}}$, which happens with probability about $T^{-1}$ (or $(i+j)^{-1}$ around $x=y$).
  The fourth one comes from the fact that the correlation between
  $D^{i,j}_{n,m}$ and $D^{i,l}_{n,m}$ decays at least as $|l-j|^{-1}$, which in computations is seen by the apparition of a kernel $p_{(l-j-1)T^{-1}}$ which is integrated over a small ball of area about $T^{-1}$.

  In some of the proofs, the contribution from the fourth part of the sum is easily shown to be equal to $0$. This, however, is not always the case, and we then have to get this factor $|l-j|^{-1}$, as well as a factor $|k-l|^{-1}$. The difficulty, in the case $i<k-1,j<l-1$ for example, is that the natural thing to do is to disintegrate with respect to the variables $X_{(i-1)T^{-1}},Y_{(j-1)T^{-1}}$, or $X_{(k-1)T^{-1}},Y_{(l-1)T^{-1}}$, but in both case we end up having to deal with a Brownian \emph{bridge} instead of a Brownian motion. Because we are lacking, for Brownian bridges, the nice estimates in the mean that we have for the Brownian motion (see Corollary \ref{coro:werner} in particular), one must avoid the apparition of such a bridge, or control how far it is from a Brownian motion, which can become tedious.

  To be perfectly rigorous, we should also deal separately with the expressions when two indices differs by exactly $1$ (e.g. $i=k+1$), or when one of them is equal to $1$ (e.g. $i=1$). Yet, these cases are always treated identically to the other ones, except for the two following things: some heat kernels such as $p_{(k-i-1)T^{-1}}$ or $p_{(i-1)T^{-1}}$ degenerate, so that they should be interpreted as Dirac measures. The steps that consist in bounding the integral on some balls  of such a kernel, by bounding the kernel itself by its maximum, should simply be replaced by the operation of bounding the integral directly by $1$.
  }
\end{remark}
\begin{proof}[Proof of Lemma \ref{le:tech1}]
  We assume $x=0$. Set \[P^{i,j}_{n,m}=D^{i,j}_{n,m}-\mathbb{E}[D^{i,j}_{n,m}|X_{(i-1)T^{-1}},Y_{(j-1)T^{-1}},Y_{jT^{-1}}].\]
  We first show that $ nm \sum_{i,j=1}^T P^{i,j}_{n,m}$ goes to $0$ sufficiently fast.

  \diams Assume $i>k$ and $j\neq l$.
  Let $\sigma_i=\sigma((X_{hT^{-1}})_{h\leq i-1}, (Y_{hT^{-1}})_{h\in \{0,\dots, T\} }, (Y_t)_{t\in [(l-1)T^{-1}, lT^{-1}]} )$. Then, from the Markov property of the Brownian motion,
  \[\mathbb{E}[ D^{i,j}_{n,m}|\sigma_i]=\mathbb{E}[ D^{i,j}_{n,m}|X_{(i-1)T^{-1}}, Y_{(j-1)T^{-1}}, Y_{jT^{-1}}],\]
  and it follows that $\mathbb{E}[P^{i,j}_{n,m}|\sigma_i]=0$.

  Besides, since $i>k$, $P^{k,l}_{n,m}$ is $\sigma_i$-measurable, so that
  \[
  \mathbb{E}[P^{i,j}_{n,m} P^{k,l}_{n,m}]=\mathbb{E}[\mathbb{E}[ P^{i,j}_{n,m}| \sigma_i]P^{k,l}_{n,m} ]=0.
  \]
  The same results of course also apply to $i<k$, $j\neq l$.

  %
  \diams Let us now look at $\mathbb{E}[(P^{i,j}_{n,m})^2]$, which is smaller than $2\mathbb{E}[(D^{i,j}_{n,m})^2]$. Let $\bar{Y}:[0,2]\to \mathbb{R}^2$ be a Brownian motion that extends $Y$, in the sense that $\bar{Y}_t=Y_t$ for $t\in[0,1]$. We then extend the definition of $Y^i$ to
  $i\in \{T+1,\dots, 2T\}$ by setting
  $Y^i=\bar{Y}_{|(i-1)T^{-1},iT^{-1}}$.
  Then, for $i,j\in \{0,\dots, T\}$, $D^{i,j}_{n,m}$ is equal in distribution to $D^{1,j+i-1}_{n,m}$ (including for $j+i-1>T$), so we can restrict ourselves to study the case $i=1$.


  Let $\alpha=\frac{1}{2}-\epsilon$ and $n_X=\|X\|_{\mathcal{C}^\alpha, [0,T^{-1}]}$ be the $\alpha$-H\"older norm of $X$ restricted to the interval $[0,T^{-1}]$, and $n^j_Y=\|Y\|_{\mathcal{C}^\alpha, [(j-1)T^{-1},jT^{-1}]}$.

   We remark that $\D^{1,j}_{n,m}$ is included on the intersection of the two balls $B_X=B(0, T^{-\alpha} n_X )$ and
  $B^j_Y=B(Y_{(j-1)T^{-1}}, T^{-\alpha} n^j_Y )$, which must therefore must be non-empty for $D^{i,j}_{n,m}$ to be non-zero.
  Thus, we have, for arbitrary $\epsilon>0$ and $p>1$,
  \begin{align*}
  \mathbb{E}[(P^{1,j}_{n,m})^2]
  &\leq 2 \mathbb{E}[|B_X|^2 \mathbbm{1}_{n_X\geq T^\epsilon }  ]+2 \mathbb{E}[|B_Y|^2 \mathbbm{1}_{n^j_Y\geq T^\epsilon }  ]
  +2
  \mathbb{E}[(D^{1,j}_{n,m})^2 \mathbbm{1}_{n_X\leq T^{\epsilon}, n^j_Y\leq T^{\epsilon}}]\\
  &\leq 4 \pi^2T^{-4\alpha}\mathbb{E}[\|X\|_{\mathcal{C}^\alpha}^2   \mathbbm{1}_{\|X\|_{\mathcal{C}^\alpha}\geq T^\epsilon}]
  +2 \mathbb{E}[(D^{1,j}_{n,m})^2 \mathbbm{1}_{n_X\leq T^{\epsilon}, n^j_Y\leq T^{\epsilon}}]\\
  &\leq 4\pi^2T^{-4\alpha} \mathbb{E}[\|X\|_{\mathcal{C}^\alpha}^4]^{\frac{1}{2}}\mathbb{E}[\|X\|_{\mathcal{C}^\alpha}^p]^{\frac{1}{2}}T^{-\frac{p\epsilon}{2}}
  +2 \mathbb{E}[(D^{1,j}_{n,m})^2 \mathbbm{1}_{n_X\leq T^{\epsilon}, n^j_Y\leq T^{\epsilon}}].
  \end{align*}
  Since $T$ is more than a positive power of $n\vee m$, we can choose $p$ such that $T^{-\frac{p\epsilon}{2} }=o(n^{-2}m^{-2}T^{-2-4\alpha})$.

  In order to control the last term, we apply a scaling and we disintegrate with respect to the value of $Y_{(j-1)T^{-1}}$. For $t\in[0,1]$, let $\tilde{X}_t= \sqrt{T} X_{tT^{-1}}$ and $\tilde{Y}_t= \sqrt{T} Y_{(j-1+t)T^{-1}}$. Then, $\|\tilde{X}\|_{\mathcal{C}^\alpha}=T^{\frac{1}{2}-\alpha} n_X=T^{\epsilon} n_X$.

  For $j>1$,
  \begin{align*}
  \mathbb{E}_{0,y}[(D^{1,j}_{n,m})^2 \mathbbm{1}_{n_X\leq T^{\epsilon}, n^j_Y\leq T^{\epsilon}} ]
  &= \int_{\mathbb{R}^2} p_{(j-1)T^{-1}}(y,z)\mathbb{E}_{0,y}\big[(D^{1,j}_{n,m})^2 \mathbbm{1}_{n_X\leq T^{\epsilon}, n^j_Y\leq T^{\epsilon}}\big| Y_{(j-1)T^{-1}}=z \big]\d z \\
  &= \int_{\mathbb{R}^2} p_{(j-1)T^{-1}}(y,z)\mathbb{E}_{0,z}\big[(D^{1,1}_{n,m})^2 \mathbbm{1}_{n_X\leq T^{\epsilon}, n^1_Y\leq T^{\epsilon}} \big]\d z \\
  &\leq \int_{\mathbb{R}^2} p_{(j-1)T^{-1}}(y,z) \mathbbm{1}_{z\leq T^{\epsilon} T^{-\alpha}} \mathbb{E}_{0,z}\big[(D^{1,1}_{n,m})^2 \big]\d z\\
  &\leq \frac{T^{-1+4\epsilon} }{2(j-1)T^{-1}} \sup_{z\in \R^2}  \mathbb{E}_{0,z}\big[(D^{1,1}_{n,m})^2 \big]\\
  &\leq \frac{T^{-3+4\epsilon} }{2(j-1)T^{-1}} \sup_{z\in \R^2}  \mathbb{E}_{0,z}\big[(D^{X,Y}_{n,m})^2 \big]\\
  &\leq C \log(T)^c \frac{T^{-3+4\epsilon} }{2(j-1)T^{-1}} n^{-2}m^{-2}.
  \end{align*}
  For the last inequality, we used Lemma \ref{le:globalBound}.

  If follows that \[\sum_{j=1}^{2T} \mathbb{E}_{0,y}[(D^{1,j}_{n,m})^2 \mathbbm{1}_{n_X\leq T^{\epsilon}, n^j_Y\leq T^{\epsilon}} ]\leq C' \log(T)^{c+1} T^{-1+4\epsilon} n^{-2}m^{-2}\leq C'' T^{-1+5\epsilon} n^{-2}m^{-2}.\]
  %
  From Cauchy--Schwarz inequality, we also get a similar bound on the expectations $\mathbb{E}_{0,y}[P^{i,j}_{n,m}P^{i,j+1}_{n,m}]$

  \diams We now consider $\mathbb{E}_{0,y}[P^{i,j}_{n,m}P^{i,l}_{n,m}]$, assuming $l\geq j+2$.
  Once again, we can assume $i=1$. We set $\bar{D}^{1,j}_{n,m}=\mathbb{E}[D^{1,j}_{n,m}|Y_{(j-1)T^{-1}}, Y_{jT^{-1}}] $.
  An elementary computation gives
  \[\mathbb{E}[P^{1,j}_{n,m}P^{1,l}_{n,m} ]\leq \mathbb{E}[D^{1,j}_{n,m}D^{1,l}_{n,m}]+
  \mathbb{E}[\bar{D}^{1,j}_{n,m}\bar{D}^{1,l}_{n,m}]
  .\]
  We treat the first expression, the second one can be bounded in an identical way.

  For all $p$, there exists $C_p$ such that for all $n$,
  \[
  \mathbb{E}[D^{1,j}_{n,m}D^{1,l}_{n,m}]\leq \mathbb{E}[D^{1,j}_{n,m}D^{1,l}_{n,m}\mathbbm{1}_{\max(n_X,n_Y,n_Y')\leq T^\epsilon}]+ C_p T^{p\epsilon},
  \]

  For $t\in[0,1]$, we set
  \[
  U=Y_{(j-1)T^{-1}}, \quad V=Y_{(l-1)T^{-1}}-Y_{jT^{-1}},\]
  \[\hat{X}_t=\sqrt{T} X_{tT^{-1}}-x, \quad \hat{Y}_t=\sqrt{T}(Y_{(j-1+t)T^{-1}}-U-y),\quad \hat{Z}_t=\sqrt{T}( Y_{(l-1+t)T^{-1}}-V-\hat{Y}_1-U-y   ).\]
  Then, $\hat{X}$, $\hat{Y}$ and $\hat{Z}$ are three Brownian motions starting from $0$, the five random variables $(U,V,\hat{X},\hat{Y},\hat{Z})$ are independent, and we have
  \[
  D^{1,j}_{n,m}D^{1,l}_{n,m}=
  T^{-2}| \D^{\hat{X}}_n\cap (\D_m^{\hat{Y}}+(x+U)\sqrt{T}     )|| \D^{\hat{X}}_n\cap (\D_m^{\hat{Z}}+(x+U+V)\sqrt{T}+\hat{Y}_1 )|.
  \]
  Disintegrating with respect to $(U,V)$, we get
  \begin{align*}
  \mathbb{E}_{0,y}[D^{1,j}_{n,m}D^{1,l}_{n,m}\mathbbm{1}_{\max(n_X,n_Y,n_Y')\leq T^\epsilon}]
  &=T^{-2}\int_{\mathbb{R}^2\times \mathbb{R}^2}
  p_{(j-1)T^{-1}}(x,u) p_{(l-j-1)T^{-1}}(0,v)\\
  &\hspace{-3.5cm}
  \mathbb{E}_{0,y}[| \D^{\hat{X}}_n\cap (\D_m^{\hat{Y}}+u\sqrt{T}     )|| \D^{\hat{X}}_n\cap (\D_m^{\hat{Z}}+(u+v)\sqrt{T}+\hat{Y}_1 )|\mathbbm{1}_{\max(n_X,n_Y,n_Y')\leq T^\epsilon}]\d u \d v.
  \end{align*}
  For the last expectation to be different from $0$, $|u|\sqrt{T}$ must be smaller than $T^{\epsilon}(n_X+n_Y)$, and therefore $|u|$ must be smaller than $2T^{-\frac{1}{2}+2\epsilon}$. Besides, $T^{\epsilon} |u+v|\sqrt{T}+\hat{Y}_1 $ must also be smaller than $n_X+n'_Y$, and therefore $|v|$ must be smaller than $5 T^{-\frac{1}{2}+2\epsilon}$.

  We get
  \begin{align*}
  \mathbb{E}_{0,y}[D^{1,j}_{n,m}D^{1,l}_{n,m}\mathbbm{1}_{\max(n_X,n_Y,n_Y')\leq T^\epsilon}]
  &= T^{-2}\int_{B(0,2  T^{-\frac{1}{2}+2\epsilon})\times B(0,5 T^{-\frac{1}{2}+2\epsilon}) }\hspace{-0.5cm}
  p_{(j-1)T^{-1}}(0,u) p_{(l-j-1)T^{-1}}(0,v)\\&\hspace{-3.5cm}
  \mathbb{E}_{0,y}[| \D^{\hat{X}}_n\cap (\D_m^{\hat{Y}}+u\sqrt{T}     )|| \D^{\hat{X}}_n\cap (\D_m^{\hat{Z}}+(u+v)\sqrt{T}+\hat{Y}_1 )|\mathbbm{1}_{\max(n_X,n_Y,n_Y')\leq T^\epsilon}]\d u \d v,
  \end{align*}
  and we can now eliminate the troubles related to the appearance of $\hat{Y}_1$ with a simple Cauchy--Schwarz inequality:
  \begin{align*}
  \mathbb{E}_{0,y}[D^{1,j}_{n,m}D^{1,l}_{n,m}&\mathbbm{1}_{\max(n_X,n_Y,n_Y')\leq T^\epsilon}]
  \leq T^{-2}\int_{B(0,2  T^{-\frac{1}{2}+2\epsilon})\times B(0,5 T^{-\frac{1}{2}+2\epsilon}) } \hspace{-0.5cm}
  p_{(j-1)T^{-1}}(x,u) p_{(l-j-1)T^{-1}}(0,v)\\
  &\hspace{-1cm}
  \mathbb{E}_{0,y}[| \D^{\hat{X}}_n\cap (\D_m^{\hat{Y}}+u\sqrt{T})|^2]^{\frac{1}{2}}
  \mathbb{E}_{0,y}[| \D^{\hat{X}}_n\cap (\D_m^{\hat{Z}}+(u+v)\sqrt{T}+\hat{Y}_1 )|^2]^{\frac{1}{2}}\d u \d v\\
  &\leq C T^{-2} T^{-1+4\epsilon} \sup_z\mathbb{E}_{0,z}[(D^{\hat{X},\hat{Y}}_{n,m} )^2] \int_{B(0,5 T^{-\frac{1}{2}+2\epsilon})  } p_{(l-j-1)T^{-1}}(0,v) \d v\\
  &\leq C' T^{-3+4\epsilon} \log(T)^c n^{-2}m^{-2} \frac{T^{4\epsilon}}{l-j-1},
  \end{align*}
  using again Lemma \ref{le:globalBound} for the last inequality.

  We now finally sum over $i$ and $j$. We have
  \begin{align}
  n^2m^2 \mathbb{E}\Big[\Big(
  \sum_{i,j} P^{i,j}_{n,m}\Big)^2\Big]
  &\leq 3n^2m^2\sum_{i,j=1}^T \mathbb{E} [(P^{i,j}_{n,m})^2]+ n^2m^2\sum_{\substack{i,j,k=1\\ |j-k|>1 }}^T \mathbb{E} [P^{i,j}_{n,m}P^{i,k}_{n,m}]\nonumber\\
  &\leq C T^{-1+4\epsilon} \log(T)^c+C \log(nm)^k T^{-1+8\epsilon} \sum_{l=1}^{2T} \frac{1}{l}\nonumber\\
  &\leq C' \log(T)^{c+1} T^{8\epsilon}  T^{-1}.
  \label{eq:conv0}
  \end{align}
  %
  This is not exactly what we wanted because of the variable $Y_{jT^{-1}}$ appearing in the definition of $P^{i,j}_{n,m}$. In order to conclude, we set
  \[ R^{i,j}_{n,m}= D^{i,j}_{n,m}-\mathbb{E}[ D^{i,j}_{n,m}| X_{(i-1)T^{-1}}, X_{iT^{-1}}, Y_{(j-1)T^{-1}}]
  .\]
  From the symmetry between $X$ and $Y$, the bound \eqref{eq:conv0} also holds with $R^{i,j}_{n,m}$ replacing $P^{i,j}_{n,m}$. Since the conditional expectation is a projection in $L^2$, we also have
  \begin{align*}
  & \Big\|nm \sum_{i,j} (\mathbb{E}[D^{i,j}_{n,m} | X_{(i-1)T^{-1}}, Y_{(j-1)T^{-1}}, Y_{j T^{-1}}  ]- \mathbb{E}[D^{i,j}_{n,m} | X_{(i-1)T^{-1}}, Y_{(j-1)T^{-1}}  ])\Big\|_{L^2}^2\\
  &=
\Big\|   \mathbb{E}\Big[ nm\sum_{i,j} R^{i,j}_{n,m} \Big| X_{(i-1)T^{-1}}, Y_{(j-1)T^{-1}}, Y_{j T^{-1}}\Big]\Big\|_{L^2}^2\\
  &\leq C T^{-1+\epsilon}.
  \end{align*}
  We conclude by combining this  with \eqref{eq:conv0}.
\end{proof}

\subsection{Elimination of the indices $n,m$}
Our next step is to prove the following lemma.
\begin{lemma}
  \label{le:tech2}
  For all $c>0$, there exists a constant $C$ such that for all $n,m,T$ with $T<m\leq n<T^{c^{-1}}$,
  \[ \sup_{x,y\in \R^2} \mathbb{E}_{x,y}\Big[\Big(
  \sum_{i,j=1}^T \big( nm \mathbb{E}[D^{i,j}_{n,m}|X_{(i-1)T^{-1}},Y_{(j-1)T^{-1}}] -T^{-1} L_{\sqrt{T}(Y_{(j-1)T^{-1}}-X_{(i-1)T^{-1}})}
  \big)\Big)^2\Big]\leq C T^4m^{-1}.
  \]
\end{lemma}
We first need some preparation. We start with some estimations in mean for $D^{X,Y}_{n,m}$.
\subsubsection{Asymptotic in mean}
\label{subsub}
The following result can be found in \cite[Lemme 5]{Werner}.\footnote{Remark that our normalisation is different from the one in \cite{Werner}, because we consider the winding as an integer instead of a continuous angle.}
\begin{lemma}
  \label{le:wendelin}
  Let $f_n(x)=\mathbb{P}_0(\theta_X(x)=n)$ and
  \[\Theta^1=\{\phi:\mathbb{R}^2\setminus \{0\}\to \R: \forall k>0, \int_{\mathbb{R}^2} |\phi(x)|^k \d x<\infty\}.\]
  Then,
  \begin{itemize}
  \item There exists $n_0$ and $\phi\in \Theta^1$ such that for all $n\geq n_0$, $n^2f_n\leq \phi$. 
  \item For all $x\neq 0$, as $n\to \infty$, \[n^2f_n(x)\longrightarrow l(x)\coloneqq\frac{1}{2\pi} \int_0^1 p_s(0,x) \d s
  .\]
  \end{itemize}
\end{lemma}
We are first going to show that, when $f_n$ is replaced with $g_n=\mathbb{P}_0(\theta_X(z)\geq n)$, not only can we bound $n g_n$, but we can also get a bound on the convergence rate.
\begin{lemma}
\label{le:techMean}
  Let $g_n(x)=\mathbb{P}_0(\theta_X(z)\geq n)$.
  For $n\geq 2$ and $r>0$, let also
  \[C_n= 2\pi \ln\Big(1+\frac{1}{2\pi n-1} \Big).\]
  Then, there exists $C$ such that for all $z\in \R^2$ and $n\geq 2$,

  \[ 0\leq C_nl(z)-g_n(z) \leq
  \left\{ \begin{array}{ll} Cn^{-3} |z|^{-2}  & \mbox{when } |z|\leq 1,\\  Cn^{-3}e^{-\frac{|z|}{3}} &\mbox{when } |z|\geq 1. \end{array}\right.
  \]

  In particular, there exists $C'$ such that for all $z\in \R^2$ and $n\geq 2$,
  \[\big|n g_n(z)- l (z)\big|
  \leq
  \left\{ \begin{array}{ll} C' (n^{-1} l (z)+n^{-2} |z|^{-2})  & \mbox{when } |z|\leq 1,\\  C'(n^{-1} l (z)+n^{-2}e^{-\frac{|z|}{3}} ) &\mbox{when } |z|\geq 1. \end{array}\right.
  \]
\end{lemma}
\begin{proof}
  The proof is very similar to, though simpler than, our proof of Lemma 3.1 in \cite{LAWA}.
  We start with the following formula, which is Theorem 5.2 in \cite{Mansuy}.\footnote{In \cite{Mansuy}, the formula is given for the continuous determination $\tilde{\theta}$ of the angle along the trajectory. The relation between $\tilde{\theta}$ and $\theta_X$ is given by $ \tilde{\theta}(z)-2\pi \theta \in [-\pi,\pi]$. }
  \[
  \mathbb{P}_r\big(\theta(0)\geq n\big| |B_1|=\rho \big)=\frac{1}{2\pi^2 I_0(r\rho)} \int_{2\pi (n-1)}^{2\pi n} \int_0^\infty e^{-r\rho \cosh(t)} \frac{x}{x^2+t^2} \d t \d x,\]
  from which we deduce, integrating back with respect to $\rho$,
  \[
  g_n(z)=\frac{1}{2\pi^2}\int_0^\infty
  \frac{\rho}{ I_0(|z|\rho)} \int_{0}^{2\pi} p_1(|z|,\rho e^{iu} ) \d u  \int_{2\pi (n-1)}^{2\pi n} \int_0^\infty e^{-|z|\rho \cosh(t)} \frac{x}{x^2+t^2} \d t \d x \d \rho.\]
  We set  \[\tilde{ l }(z)=\frac{1}{4\pi^3}\int_0^\infty \frac{\rho}{I_0(|z|\rho)} \int_0^{2\pi} p_1(|z|,\rho e^{iu} ) \d u \int_0^\infty e^{-|z|\rho \cosh(t)} \d t \d \rho,\]
  where $I_0$ is the modified Bessel function of the first kind, and $p_1(r,\rho e^{iu})$ is defined by the canonical identification between $\mathbb{C}$ and $\R^2$.
  which as we will deduce later is actually equal to $ l (z)$. First, we show that the first part of the lemma holds with $\tilde{ l }$ replacing $ l $.

  The bound $0\leq C_n\tilde{ l }(z)-g_n(z)$ is simply obtain by using $\frac{x}{x^2+t^2}<\frac{1}{x}$. For the other inequality, we first assume $|z|\leq 1$ and we use the bounds
  \[\frac{1}{x}-\frac{x}{x^2+t^2}=\frac{t^2}{x(x^2+t^2)}\leq \frac{t^2}{x^3}, \quad p_1(x,y)\leq \frac{1}{2\pi}, \quad I_0(|z|\rho)\geq 1.\]
  We obtain
  \begin{align*}
  C_n\tilde{ l }(z)-g_n(z)&\leq \frac{1}{2\pi^2}  \int_{2\pi (n-1)}^{2\pi n} x^{-3}\d x
   \int_0^\infty\int_0^\infty
  \rho   e^{-|z|\rho \cosh(t)} t^2 \d t \d \rho\\
  &\leq  C n^{-3} \int_0^\infty \frac{t^2}{|z|^2\cosh(t)^2} \d t\\
  &\leq C'n^{-3} |z|^{-2}.
  \end{align*}

  This would also work for $|z|\geq 1$, but the bound would not be sufficient for our purpose. In that case, we split the integral on $\rho$ at $1$. For $\rho\leq 1$, we use $p_1(|z|, \rho e^{iu})\leq p_1(0,|z|-1)$. For $\rho\geq 1$, we use $e^{-|z|\rho\cosh(t)}\leq e^{-\frac{|z|}{3}} e^{-\frac{\rho}{3}}e^{-\frac{t}{3}}$.
  We get
  \begin{align*}
  C_n\tilde{ l }_{|z|}-g_n(z)&\leq \frac{1}{2\pi^2} \int_{2\pi(n-1)}^{2\pi n} x^{-3} \d x \Big(
  2 \pi p_1(0,|z|-1) \int_0^1 \rho \int_0^\infty e^{-|z|\rho \cosh(t)} t^2 \d t\d \rho\\
  &\hspace{7.5cm}+e^{-\frac{|z|}{3}}
  \int_1^\infty \rho e^{-\frac{\rho}{3}} \int_0^\infty e^{- \frac{t}{3}} t^2 \d t\d \rho
  \Big)\\
  &\leq C \int_{2\pi(n-1)}^{2\pi n} x^{-3} \d x \ e^{-\frac{|z|}{3}}.
  \end{align*}
  Now, we can deduce that $ l =\tilde{ l }$. Indeed, for all $z\neq 0$, we know that \[n^2(g_n(z)-g_n(z+1))=f_n(z)n^2\underset{n\to \infty}\longrightarrow l (z)\] from Lemma \ref{le:wendelin}, but we also know that \[n^2(g_n(z)-g_n(z+1))=n^2(C_n\tilde{ l }(z)-C_{n+1}\tilde{ l }(z)+o(n^{-3}) )\underset{n\to \infty}\longrightarrow_{n\to \infty} \tilde{l}(z).\]
\end{proof}

\begin{corollary}
  \label{coro:werner}
  Let $A^{X,Y}_{n,m}=|\{z\in \R^2: \theta_X(z)\geq n,  \theta_Y(z)\geq m\}|$. Then,
  \[n^2m^2\mathbb{E}_{x,y}[A^{X,Y}_{n,m}]\underset{n,m\to \infty}\longrightarrow L_{y-x}\coloneqq\frac{1}{4\pi^2} \int_0^1\int_0^1 p_{s+u}(x,y) \d u \d s, \mbox{ and } nm\mathbb{E}_{x,y}[D^{X,Y}_{n,m}]\underset{n,m\to \infty}\longrightarrow L_{y-x},\]
  with convergence uniform in $x$ and $y$.

  There exists a constant $C$ such that for all $n,m\geq 1$,
  \[
  \sup_{x,y\in \R^2} |nm\mathbb{E}_{x,y}[D^{X,Y}_{n,m}]-L_{y-x}|\leq C (n^{-\frac{1}{2}} + m^{-\frac{1}{2}}).
  \]

  Besides, there exists an integrable function $\phi$ and $n_0$ such that for all $n,m>n_0$ and for all $y\in \R^2$, $nm\mathbb{E}_{y}[D^{X,Y}_{n,m}] \leq \phi(y)$.
\end{corollary}
\begin{proof}
  From translation invariance, we can assume $x=0$.
  For the first convergence, it suffices to write the left hand side as
  \[ \int_{\mathbb{R}^2} n^2\mathbb{P}_{0,y}(\theta_X(z)=n )m^2\mathbb{P}_{0,y}(\theta_Y(z-y)=m ) \d z.\]
  The expression under the integral is dominated (for $n_0$ large enough) by $\phi(z)\phi(z-y)$, where $\phi$ is given by Lemma \ref{le:wendelin}. The function $z\mapsto \phi(z)\phi(z-y)$ is integrable from Cauchy--Schwarz inequality, which ensures that we can apply the dominated convergence theorem, so that the limit of the left hand side exists
  is equal to \[\int_{\mathbb{R}^2} \frac{1}{4\pi^2} \int_0^1 p_s(0,z) \d s \int_0^1 p_u(y,z) \d u \d z =\frac{1}{4\pi^2} \int_0^1\int_0^1 p_{s+u}(0,y) \d u \d s=L(y).\]

  For the uniformity, we set $\phi$ as in the previous lemma, and we write
  \begin{align*}
  \sup_y &|n^2m^2 \mathbb{E}_{0,y}[A_{n,m}]-L(y)|
  \leq\sup_y
  \int_{\mathbb{R}^2} |n^2f_n(z)m^2f_m(z-y)-l(z)l(z-y)|  \d z\\
  &\leq \sup_y \int_{\mathbb{R}^2} l(z)|m^2f_m(z-y)-l(z-y)|  \d z+\sup_y
  \int_{\mathbb{R}^2} |n^2f_n(z)-l(z)|m^2f_m(z-y)\d z\\
  &\leq  \sup_y \Big(\int_{\mathbb{R}^2} \phi(z)^2 \d z \int_{\mathbb{R}^2} |m^2f_m(z-y)-l(z-y)|^2  \d z\Big)^{\frac{1}{2}}\\&+\sup_y
  \Big(\int_{\mathbb{R}^2} |n^2f_n(z)-l(z)|^2\d z \int_{\R^2} (\phi(z-y))^2 \d z \Big)^{\frac{1}{2}}\\
  &\leq C \Big(\int_{\mathbb{R}^2} |m^2f_m(z)-l(z)|^2  \d z\Big)^{\frac{1}{2}}
  +C \Big(\int_{\mathbb{R}^2} |n^2f_n(z)-l(z)|^2  \d z\Big)^{\frac{1}{2}}\\
  &\hspace{-0.2cm}\underset{n,m\to \infty}\longrightarrow 0.
  \end{align*}

  For the second convergence, we first show that there exists $\phi\in \Theta^1: ng_n\leq \phi$. Indeed, let $g_n(x)=\mathbb{P}_0(\theta_X(x)\geq n)=\sum_{k=n}^\infty f_n(x)$. Let $\phi$ and $n_0$ be given as in Lemma \ref{le:wendelin}. Then, for all $n\geq n_0$,
  \[n g_n\leq n\sum_{k=n}^\infty k^{-2} \phi\underset{n\to \infty}\sim \phi .\]
  Thus, for $n_1$ sufficiently large, for all $n\geq n_1$, $ng_n\leq 2\phi$. With Lemma \ref{le:techMean}, we know that $ng_n$ converges pointwise toward $\ell$, and we then repeat the argument above.

  Now we need to bound the remaining part.
  We have
  \[
  \sup_y |nm \mathbb{E}_{0,y}[D_{n,m}]-L(y)|\leq C  \Big(\int_{\mathbb{R}^2} |mg_m(z)-l(z)|^2  \d z\Big)^{\frac{1}{2}}
  +C \Big(\int_{\mathbb{R}^2} |ng_n(z)-l(z)|^2  \d z\Big)^{\frac{1}{2}},
  \]
  and, for all $\epsilon>0$,
  \begin{align}
  \int_{\mathbb{R}^2} |ng_n(z)-l(z)|^2  \d z
  &\leq
  \int_{B(0,\epsilon)} l(z)^2\d z+\int_{B(0,\epsilon)}  n^2 \d z
  +\int_{B(0,1)\setminus B(0,\epsilon) } (n^{-1} l(z) +n^{-2} |z|^{-2})^2 \d z\nonumber\\
  &\hspace{5.5cm}+\int_{\R^2\setminus B(0,1) } (n^{-1} l(z) +n^{-2} e^{-\frac{|z|}{3} } )^2 \d z.
  \label{eq:bornInt}
  \end{align}
  In order to bound
  $\int_{B(0,\epsilon)} l(z)^2\d z$, we first remark that
  \begin{align*}
  \int_{B(0,\epsilon)} l(z)^2\d z&=\frac{1}{4\pi^2} \int_0^1 \int_0^1 \int_{B(0,\epsilon)} p_s(0,x)p_t(0,x) \d x\d t\d s\\
  &<\frac{1}{4\pi^2}   \int_0^1 \int_0^1 p_{s+t}(0,0) \d s\d t\\
  &<\infty.
  \end{align*}
  Then, by applying the changes of variables $y=\epsilon^{-1} x$, $u=\epsilon^{-2} t$, $v=\epsilon^{-2} s$, we get
  \begin{align*}
  \int_{B(0,\epsilon)} l(z)^2\d z&=\frac{\epsilon^2}{4\pi^2} \int_{B(0,1)} \int_0^{\epsilon^{-2}} \int_0^{\epsilon^{-2}} p_u(0,y)p_v(0,y) \d y\d u\d t\\
  &\leq \frac{\epsilon^2}{4\pi^2} \int_{B(0,1)} \int_0^{1} \int_0^{1} p_u(0,y)p_v(0,y) \d y\d u\d t
  +\frac{\epsilon^2}{4\pi^2} \int_{B(0,1)} \int_1^{\epsilon^{-2}} \int_1^{\epsilon^{-2}} \frac{1}{uv} \d y\d u\d t\\
  &\leq C' \epsilon^2 \log(\epsilon)^2.
  \end{align*}
  Going back to \eqref{eq:bornInt}, we obtain
  \begin{align*}
  \int_{\mathbb{R}^2} |ng_n(z)-l(z)|^2  \d z
  &\leq C \epsilon^2 \log(\epsilon)^2 +\pi n^2\epsilon^2
  +C n^{-2}+ C n^{-4} \epsilon^{-2}
  +C n^{-2}+Cn^{-4}.
  \end{align*}
  with $\epsilon=n^{-\frac{3}{2}}$, we obtain
  \[\int_{\mathbb{R}^2} |ng_n(z)-l(z)|^2  \d z\leq C' n^{-1}, \]
  which concludes the proof of the second point.

  For the last point, simply remark that $nm\mathbb{E}_{0,y}[D_{n,m}]$ is bounded by $\int_{\R^2} 4\phi(z)\phi(y-z)\d z$, which is integrable over the plane since
  \[\int_{\R^2\times \R^2} \phi(z)\phi(z-y)\d z \d y= \Big(\int_{\R^2} \phi(z)\d z \Big)^2<\infty.\]
\end{proof}

\subsubsection{Proof of Lemma \ref{le:tech2}}
Let us recall this Lemma.
\begin{lemma*}
  For all $c>0$, there exists a constant $C$ such that for all $n,m,T$ with $T<m\leq n<T^{c^{-1}}$, for all $x,y\in \R^2$,
  \[ \mathbb{E}_{x,y}\Big[\Big(
  \sum_{i,j=1}^T \big( nm \mathbb{E}[D^{i,j}_{n,m}|X_{(i-1)T^{-1}},Y_{(j-1)T^{-1}}] -T^{-1} L_{\sqrt{T}(Y_{(j-1)T^{-1}}-X_{(i-1)T^{-1}})}
  \big)\Big)^2\Big]\leq C T^4m^{-1}.
  \]
\end{lemma*}
\begin{proof}
  Remark that
  \[ \mathbb{E}[D^{i,j}_{n,m}|X_{(i-1)T^{-1}},Y_{(j-1)T^{-1}}]= T^{-1} \mathbb{E}_{\sqrt{T}X_{(i-1)T^{-1}},\sqrt{T}Y_{(j-1)T^{-1}}}[D^{\hat{X},\hat{Y}}_{n,m}]. \]
  It follows that the left-hand side in the lemma is smaller than \[T^3 \sup_{x,y\in \R^2} |nm \mathbb{E}_{x,y}[D^{\hat{X},\hat{Y}}_{n,m}]- L_{x-y}|^2,\] which we know to be smaller than $CT^3 m^{-1}$ by Corollary \ref{coro:werner}.
\end{proof}

\subsection{Large $T$ limit}
At this stage,
we have made an important step, since we are now free from the variables $n$ and $m$ (and from any winding function).
Our goal now is to replace the sum with an integral,
for which we need to bound
\[
\sum_{i,j=0}^{T-1} T^{-1} L_{T^{\frac{1}{2}}
(Y_{jT^{-1}}-X_{iT^{-1}} ) }
- T^2\int_0^1\int_0^1 T^{-1} L_{T^{\frac{1}{2}}(Y_s-X_t)} \d s\d t.
\]
Remark that we have shifted the indices $i$ and $j$ by $1$ for convenience.
\begin{lemma}
  \label{le:tech3}
  For all $\epsilon>0$, there exists a constant $C$ such that for all $K$, for all function $f\in \mathcal{C}^2_c(\R^2)$, supported on $B(0,K)$,
  \[\mathbb{E}\Big[
  \Big( T^{-1}\sum_{i,j=0}^{T-1} f(T^{\frac{1}{2}}X_{iT^{-1}}, T^{\frac{1}{2}}Y_{jT^{-1}}) -T\int_0^1\int_0^1 f(\sqrt{T}X_s,\sqrt{T}X_t )\Big)^2\Big]^{\frac{1}{2}}\leq C K \|f\|_\infty T^{-\frac{1}{2}+\epsilon} .
  \]
\end{lemma}
\begin{proof}
  Let
  \[
  G_{i,j}^{s,t}=f(T^{\frac{1}{2}}X_{iT^{-1}+s}, T^{\frac{1}{2}}Y_{jT^{-1}+t}), \quad
  H_{i,j}^{s,t}=G_{i,j}^{s,t}-G_{i,j}^{0,0} 
  \]
  and
  \begin{align*}
  Z_{i,j}&= T^2 \int_0^{T^{-1}} \int_0^{T^{-1}} f(T^{\frac{1}{2}}X_{iT^{-1}+s}, T^{\frac{1}{2}}Y_{jT^{-1}+t})\d s \d t-f(T^{\frac{1}{2}}X_{iT^{-1}}, T^{\frac{1}{2}}Y_{jT^{-1}})\\
  &=T^2\int_0^{T^{-1}} \int_0^{T^{-1}} H_{i,j}^{s,t} \d s\d t.
  \end{align*}
  so that the right-hand side in the lemma is $T^{-\frac{1}{2}-\epsilon}\sum_{i,j=1}^T Z_{i,j} $.
  Our goal is therefore to show that, for all $\epsilon>0$, there exists a constant $C$ such that for $T$ large enough,
  \[\sum_{i,j,k,l} \mathbb{E}[Z_{i,j}Z_{k,l}]\leq C T^{1+\epsilon}.\]

  The proof is rather long, and we skip some details. 

  Let $P^1_s f :(x,y)\mapsto \int_{\R^2} p_t(x,z) f(z,y)$ and $P^2_s f :(x,y)\mapsto \int_{\R^2} p_t(y,z) f(x,z)$. Let also
  $\nabla^1$ (resp. $\nabla^2$) be the differential with respect to the first (resp. second) variable,
  $\Delta^1$ (resp. $\Delta^2$) be the Laplacian with respect to the first (resp. second variable):
  \[\Delta^1 f:(x,y)\mapsto \Delta(z\mapsto f(z,y) )(x), \quad \Delta^2 f:(x,y)\mapsto \Delta(z\mapsto f(x,z) )(z) .\]

  From the It\^o formula (applied first to $Y$ and then to $X$), we have
  \begin{align}
  H_{i,j}^{s,t}=&f(T^{\frac{1}{2}}X_{iT^{-1}+s},T^{\frac{1}{2}}Y_{jT^{-1}+t })
  -f(T^{\frac{1}{2}}X_{iT^{-1}+s},T^{\frac{1}{2}}Y_{jT^{-1}})\nonumber\\
  &+f(T^{\frac{1}{2}}X_{iT^{-1}+s},T^{\frac{1}{2}}Y_{jT^{-1} })
  -f(T^{\frac{1}{2}}X_{iT^{-1}},T^{\frac{1}{2}}Y_{jT^{-1} })\nonumber\\
  =&\int_0^t T^{\frac{1}{2}} \nabla^2 \big(f(T^{\frac{1}{2}}X_{iT^{-1}+s},T^{\frac{1}{2}}Y_{jT^{-1}+\tau }
  -f(T^{\frac{1}{2}}X_{iT^{-1}},T^{\frac{1}{2}}Y_{jT^{-1}+\tau } )
  )\big) \d Y_{jT^{-1}+\tau }\nonumber \\
  &+\frac{1}{2}
  \int_0^t T\Delta^2 \big(f(T^{\frac{1}{2}}X_{iT^{-1}+s},T^{\frac{1}{2}}Y_{jT^{-1}+\tau } )
  -f(T^{\frac{1}{2}}X_{iT^{-1}},T^{\frac{1}{2}}Y_{jT^{-1}+\tau } )\big)
  \d \tau, \nonumber\\
  \mathbb{E}[H_{i,j}^{s,t}| X, Y_{jT^{-1} }]=&\frac{1}{2}
  \int_0^t T P^2_\tau \Delta^2 \big(f(T^{\frac{1}{2}}X_{iT^{-1}+s},T^{\frac{1}{2}}Y_{jT^{-1} } )
  - f(T^{\frac{1}{2}}X_{iT^{-1}},T^{\frac{1}{2}}Y_{jT^{-1} } )\big)
  \d \tau\hspace{-0.2cm}\label{eq:ito2}\\
  =&\frac{1}{2}\int_0^t \int_0^s T^{\frac{3}{2}}\nabla^1 P^2_\tau \Delta^2 f(T^{\frac{1}{2}}X_{iT^{-1}+\sigma},T^{\frac{1}{2}}Y_{jT^{-1} } ) \d X_{iT^{-1}+\sigma}  \d \tau\nonumber\\
  &+\frac{1}{4}\int_0^t \int_0^s T^2 \Delta^1 P^2_\tau \Delta^2 f(T^{\frac{1}{2}}X_{iT^{-1}+\sigma},T^{\frac{1}{2}}Y_{jT^{-1} } ) \d {\sigma} \d \tau,\nonumber\\
  \mathbb{E}[H_{i,j}^{s,t}| X_{iT^{-1}}, Y_{jT^{-1} }]=&
  \frac{T^2}{4}\int_0^t \int_0^s P^1_\sigma \Delta^1 P^2_\tau \Delta^2 f(T^{\frac{1}{2}}X_{iT^{-1}},T^{\frac{1}{2}}Y_{jT^{-1} } ) \d {\sigma} \d \tau.\label{eq:ito3}
  \end{align}

  \diams Let $i<j-2$, $k<l-2$, and $u,v,s,t\in [0,T^{-1}]$. Then, we get
  \begin{align}
  \mathbb{E}_{x,y}&\Big[f(T^{\frac{1}{2}} X_{iT^{-1}+s },T^{\frac{1}{2}} X_{jT^{-1}+t }  )H^{u,v}_{k,l} \Big]\nonumber\\
  =&
  \mathbb{E}_{x,y}\Big[f(T^{\frac{1}{2}} X_{iT^{-1}+s },T^{\frac{1}{2}} X_{jT^{-1}+t }  ) \mathbb{E}[H^{u,v}_{k,l}| X_{iT^{-1}+s}, Y_{jT^{-1}+t}] ]\nonumber\\
  =&\frac{T^2}{4}\int_0^u\int_0^v
  \mathbb{E}_{x,y}\Big[f(T^{\frac{1}{2}} X_{iT^{-1}+s },T^{\frac{1}{2}} X_{jT^{-1}+t }  ) \\
  &\hspace{2cm}P^1_{(j-i)T^{-1}-s +\sigma } \Delta^1 P^2_{(k-l)T^{-1}-t+\tau} \Delta^2f(T^{\frac{1}{2}} X_{iT^{-1}+s },T^{\frac{1}{2}} X_{jT^{-1}+t }  )
  \Big]\d \sigma \d \tau.\nonumber
  \end{align}
  Using
  \[ \|\partial_t p_t(0,\cdot)\|_{L^1}=\frac{1}{2\pi t} \int_0^\infty e^{-\frac{r^2}{2t}}\Big|\frac{r^2-2t}{2t^3}\Big|\d r=t^{-2}\|(\partial_t p_t(0,\cdot))_{|t=1}\|_{L^1},\]
  and using the Young's convolution inequality $\|f*g\|_{L^\infty}\leq \|f\|_{L^\infty}\|g\|_{L^1}$, we deduce that
  \[ \|\Delta^1P^1_t f\|_{L^\infty}=\| f*\Delta p_t(0,\cdot)\|_{L^\infty}=\| f*\partial_t p_t(0,\cdot)\|_{L^\infty}\leq C t^{-2} \|f\|_{L^\infty},\]
  and therefore
  \[ \|\Delta^2P^2_s \Delta^1P^1_t f\|_{L^\infty}\leq C^2 s^{-2} t^{-2} \|f\|_{L^\infty}.\]

  It follows that, for all $x,y\in \R^2$, $|f(x,y)P^2_t P^1_s \Delta^1\Delta^2  \Delta f (x,y)|\leq \mathbbm{1}_{|x-y|\leq K}C s^{-2}t^{-2} \|f\|^2_\infty  $, and therefore
  \begin{align*}
  \mathbb{E}_{x,y}&\Big[
  f(T^{\frac{1}{2}}X_{(i+u)T^{-1}},T^{\frac{1}{2}}Y_{(j+v)T^{-1}})
  P^1_{k-i-u+tT^{-1}} P^2_{l-j-v+sT^{-1}}\Delta^1\Delta^2 f (T^{\frac{1}{2}}X_{(i+u)T^{-1} },T^{\frac{1}{2}}Y_{(j+v)T^{-1} } )\Big]\\
  &  \leq \mathbb{P}_x (Y_{(j+v)T^{-1}}- X_{(i+u)T^{-1}} \leq K T^{-\frac{1}{2}}   ) C\|f\|_\infty  (k-i-u+tT^{-1})^{-2}(l-j-v+sT^{-1})^{-2}  \\
&\leq C' K^2 \|f\|_\infty^2 (i+j+1)^{-1}|k-i-1|^{-2}|l-j-1|^{-2}.
  \end{align*}
  With \eqref{eq:ito3}, we get
  \[\mathbb{E}_{x,y}\Big[f(T^{\frac{1}{2}}X_{(i+u)T^{-1}},T^{\frac{1}{2}}Y_{(j+v)T^{-1}}) Z_{k,l} \Big]\leq
  C K^{2} \|f\|_\infty^2  (i+j+1)^{-1}|k-i-1|^{-2}|l-j-1|^{-2}.
  \]
  We deduce that
  \[
  \mathbb{E}_{x,y}[H^{s,t}_{i,j}H^{u,v}_{k,l}]\leq
  C K^{2} \|f\|^2_\infty  (i+j+1)^{-1}|k-i-1|^{-2}|l-j-1|^{-2},
  \]
  and therefore
  \[
  \mathbb{E}_{x,y}[Z_{i,j}Z_{k,l}]\leq
  C K^2 \|f\|^2_\infty   (i+j+1)^{-1}|k-i-1|^{-2}|l-j-1|^{-2}.
  \]

  Summing over $i$ and $j$, we get
  \begin{align*}
  \sum_{\substack{i,j,k,l=0\\|i-j|\geq 2, |k-l|\geq 2,\\ i<k, j<l} }^{T-1}
  \mathbb{E}_{x,y}[Z_{i,j}Z_{k,l} ]
  &\leq CK^2\|f\|^2_\infty\hspace{-0.3cm}
  \sum_{\substack{i,j,k,l=0\\|i-j|\geq 2, |k-l|\geq 2,\\ i<k, j<l} }^{T-1}\hspace{-0.3cm}
 (i+j+1)^{-1}|k-i-1|^{-2}|l-j-1|^{-2}\\
  &\leq CK^2\|f\|^2_\infty T.
  \end{align*}

\diams Now we need to deal with the case  $i>k+1$, $j+1<l$, for which we are going to sacrifice an additional factor $T^\epsilon$. The reason for that is that we will end up with $(P^1_t\Delta^1f) (P^2_t\Delta^2f)$ instead of $f(P^1_tP^2_s\Delta^1\Delta^2f)$. As opposed to the latter, the former is not compactly supported anymore, and we have to mimic a compact support by treating separately the cases
  \[\{ |X_i^t-Y_j^s|\geq T^{\epsilon-\frac{1}{2}} \} \quad \mbox{and} \quad \{ |X_i^t-Y_j^s|\leq T^{\epsilon-\frac{1}{2}} \}.
  \]

  Using \eqref{eq:ito3}, we have
  \begin{align*}
  \mathbb{E}&[ \mathbbm{1}_{|X_{kT^{-1}}-Y_{jT^{-1}} |\leq T^{\epsilon-\frac{1}{2}}  }  H^{s,t}_{i,j}H^{u,v}_{k,l} ]\\
  =&\mathbb{E}[\mathbbm{1}_{|X_{kT^{-1}}-Y_{jT^{-1}} |\leq T^{\epsilon-\frac{1}{2}}  }  \mathbb{E}[H^{s,t}_{i,j}|X_{(k+1)T^{-1}}, Y] \mathbb{E}[H^{u,v}_{k,l}| X, Y_{(j+1)T^{-1}}]\\
  =&\frac{T^2}{4}\int_0^s \int_0^v \mathbb{E}\Big[\mathbbm{1}_{|X_{kT^{-1}}-Y_{jT^{-1}} |\leq T^{\epsilon-\frac{1}{2}}  }\\
  &\hspace{1cm}\big(P^1_{i-k-1+\sigma}\Delta^1 f(T^{\frac{1}{2}}X_{(k+1)T^{-1}}, T^{\frac{1}{2}}Y_{jT^{-1}+t })
  -P^1_{i-k-1+\sigma}\Delta^1 f(T^{\frac{1}{2}}X_{(k+1)T^{-1}}, T^{\frac{1}{2}}Y_{jT^{-1} })\big)\\
  &\big( P^2_{l-j-1+\sqrt{T}\nu }\Delta^2f(T^{\frac{1}{2}} X_{kT^{-1}+u },T^{\frac{1}{2}}Y_{(j+1)T^{-1} }  )-P^2_{l-j-1+\sqrt{T}\nu }\Delta^2f(T^{\frac{1}{2}} X_{kT^{-1} },T^{\frac{1}{2}}Y_{(j+1)T^{-1} }  )\big)
  \Big] \d \sigma \d \nu.
  \end{align*}
  We only treat one of the four terms obtained after developing the product,
  \begin{align*}
  E=&\mathbb{E} \big[\mathbbm{1}_{|X_{kT^{-1}}-Y_{jT^{-1}} |\leq T^{\epsilon-\frac{1}{2}}  }
  P^1_{i-k-1+\sigma}\Delta^1 f(T^{\frac{1}{2}}X_{(k+1)T^{-1}}, T^{\frac{1}{2}}Y_{jT^{-1}+t })\\
  &\hspace{6cm}P^2_{l-j-1+\nu }\Delta^2f(T^{\frac{1}{2}} X_{kT^{-1}+u },T^{\frac{1}{2}}Y_{(j+1)T^{-1} }  )\big]\\
  =&
  \mathbb{E}\big[\mathbbm{1}_{|X_{kT^{-1}}-Y_{jT^{-1}} |\leq T^{\epsilon-\frac{1}{2}}  }
  P^1_{i-k-\sqrt{T}u+\sigma}\Delta^1 f(T^{\frac{1}{2}}X_{kT^{-1}+u}, T^{\frac{1}{2}}Y_{jT^{-1}+t })\\
  &\hspace{6cm}P^2_{l-j+\nu-\sqrt{T}t }\Delta^2f(T^{\frac{1}{2}} X_{kT^{-1}+u },T^{\frac{1}{2}}Y_{jT^{-t} }  )\big].
  \end{align*}
  The three other terms are dealt with identically.

  Using again $\|\Delta^1 P^1_t f\|_\infty\leq C t^{-2}\|f\|_\infty$, we obtain
  \begin{align*}
  E&\leq \mathbb{P}(|X_{kT^{-1}}-Y_{jT^{-1}} |\leq T^{\epsilon-\frac{1}{2}}) C \|f\|^2_\infty  (i-k-1)^{-2}(l-j-1)^{-2}\\
  &\leq C\|f\|^2_\infty  T^{2\epsilon} (k+j)^{-1}  (i-k-1)^{-2}(l-j-1)^{-2  },
  \end{align*}
  and it follows that
  \[
  \mathbb{E}_{x,y}[\mathbbm{1}_{|X_{kT^{-1}}-Y_{jT^{-1}} |\leq T^{\epsilon-\frac{1}{2}}  }Z_{i,j}Z_{k,l} ]
  \leq
  C K^4\|f\|^2_\infty  T^{2\epsilon} (k+j)^{-1}  (i-k-1)^{-2}(l-j-1)^{-2  }.
  \]
  Now we sum over $i,j,k,l$:
  \begin{align*}
  \sum_{\substack{i,j,k,l\\ i>k+1,j+1<l}}\hspace{-0.5cm}\mathbb{E}_{x,y}[\mathbbm{1}_{|X_{kT^{-1}}-Y_{jT^{-1}} |\leq T^{\epsilon-\frac{1}{2}}  } Z_{i,j}Z_{k,l} ]&
  \leq
  C K^{4} \|f\|_\infty^2 T^{2\epsilon} \hspace{-0.6cm} \sum_{\substack{i,j,k,l\\ i>k+1,j+1<l}}\hspace{-0.5cm}
  (k+j)^{-1}  (i-k-1)^{-2}(l-j-1)^{-2  }\\
  &\leq C K^4 \|f\|_\infty^2 T^{1+2\epsilon} \log(T).
  \end{align*}

  Now we also need to control
  \[
  \mathbb{E}_{x,y}[\mathbbm{1}_{|X_{kT^{-1}}-Y_{jT^{-1}} |\geq T^{\epsilon -\frac{1}{2}}  } Z_{i,j}Z_{k,l} ],
  \]
  but this is simply bounded by \[8 \|f\|_\infty^2 \mathbb{P}\Big(\exists s,t \in [0,1]: |s-t|\leq T^{-1},\ |X_s-X_t|\geq \frac{ T^{\epsilon-\frac{1}{2} }}{2} \Big), \]
  which it is smaller than
  \[
  8\|f\|_\infty^2\frac{\mathbb{E}[ \sup_{s,t: |s-t|\leq T^{-1} } |X_s-X_t|^p ]}{T^{-\frac{p}{2}+\epsilon p  }},
  \]
  which itself is smaller than $C_p T^{-\epsilon p} \log(T)^{\frac{p}{2}}$ (see \cite[Lemma 1]{Fischer}). For $p>\frac{4}{\epsilon}$, we obtain
  \[
  \sum_{\substack{i,j,k,l\\ i>k+1,j+1<l}}\mathbb{E}_{x,y}[\mathbbm{1}_{|X_{kT^{-1}}-Y_{jT^{-1}} |\geq T^{\epsilon-\frac{1}{2}}  } Z_{i,j}Z_{k,l} ]\leq C \|f\|_\infty^2.
  \]
\diams
  We now need to control the terms for which $|i-j|\leq 2$ or $|k-l|\leq 2$. Fortunately, this is much simpler. First, we have
  \[\mathbb{E}[Z_{i,i}^2]\leq \|f\|_\infty^2,\]
  so that \[\sum_{\substack{i,j,k,l\\ i+j\leq 8,|j-l|,|i-k|\leq 2 }} \mathbb{E}[Z_{i,j}Z_{k,l}]\leq C \|f\|_\infty^2 T.\]
  Then, remark that
  \[
  \mathbb{E}[(H^{s,t}_{i,j})^2]\leq \|f\|^2_\infty \mathbb{P}( |X_{iT^{-1}+s}-Y_{jT^{-1}+t}|\leq KT^{-\frac{1}{2}} )\leq C \|f\|^2_\infty K^2(i+j)^{-1},\]
  from which we deduce that
  \[
  \mathbb{E}[Z_{i,j}^2]\leq C \|f\|^2_\infty K^2(i+j)^{-1}.
  \]
  For $|i-k|\leq 2, | j-l |\leq 2, i+j\geq 8 $, we have $k+l\geq i+j-4\geq \frac{i+j}{2} $, so that
  \[ \mathbb{E}[Z_{i,j}Z_{k,l}]\leq  \mathbb{E}[Z_{i,j}^2]^{\frac{1}{2}} \mathbb{E}[Z_{k,l}^2]^{\frac{1}{2}}\leq C \|f\|^2_\infty K^2(i+j)^{-\frac{1}{2}}(k+l)^{-\frac{1}{2}}\leq 2 C \|f\|^2_\infty K^2(i+j)^{-1},
  \]
  and we deduce that
  \[\sum_{\substack{i,j,k,l\\ |j-l|,|i-k|\leq 2 }} \mathbb{E}[Z_{i,j}Z_{k,l}]\leq C \|f\|_\infty^2 T.\]

\diams
  We now consider the case when $|i-k|\leq 2$ and $|j-l|\geq 2$. We assume $j\leq l$. Then,
  \begin{align*}
  \mathbb{E}[ H^{s,t}_{i,j}H^{u,v}_{k,l}]&=\frac{T}{2}\int_0^v \mathbb{E}\big[H^{s,t}_{i,j} P^2_{l-j+\nu-t }\Delta^2\big( f(T^{\frac{1}{2}}X_{kT^{-\frac{1}{2} }+u},T^{\frac{1}{2}}Y_{jT^{-\frac{1}{2} }+t}   )-f(T^{\frac{1}{2}}X_{kT^{-\frac{1}{2} } },T^{\frac{1}{2}}Y_{jT^{-\frac{1}{2} }+t}   ) \big) \big] \d \nu\\
  &\leq T \mathbb{P}(H^{s,t}_{i,j} \neq 0 ) \|f\|_\infty^2 \int_0^v (l-j+\nu-t)^{-2} \d \nu \leq C' K^2 (i+j)^{-1} \|f\|^2_\infty (l-j-1)^{-2}.
  \end{align*}
  We deduce that
  \[
  \mathbb{E}[ Z_{i,j}Z_{k,l}]
  \leq C K^2 (i+j)^{-1} \|f\|^2_\infty (l-j-1)^{-2},\]
  so that
  \[ \sum_{i,j,k,l: |i-k|\leq 2, |j-l|\geq 2}\mathbb{E}[ Z_{i,j}Z_{k,l}]\leq C K^2 \|f\|^2_\infty T.\]

  Combining all these bounds together, we get
  \[ \sum_{i,j,k,l} \mathbb{E}[ Z_{i,j}Z_{k,l}]\leq C K^2 \|f\|^2_\infty T^{1+\epsilon}.\]
\end{proof}

\begin{corollary}
  \label{coro:tech3}
  For all $\epsilon>0$, there exists a constant $C$ such that for all $f\in \mathcal{C}^0(\R^2,\R_+)$ with $\sup_{z\in \R^2} |(1+z) f(z)|<\infty$,
  \[
  T^{-\frac{1}{2}-\epsilon}\mathbb{E}\Big[
  \Big(\sum_{i,j} f(T^{\frac{1}{2}}X_{iT^{-1}}, T^{\frac{1}{2}}Y_{jT^{-1}}) -T^2\int_0^1\int_0^1 f(T^{\frac{1}{2}}X_s,T^{\frac{1}{2}}X_t )\Big)^2\Big]^{\frac{1}{2}}\leq C \sup_{z\in \R^2} (1+|z|) f(z).
  \]
\end{corollary}
\begin{proof}
  First, remark that Lemma \ref{le:tech3} holds with the condition $f\in \mathcal{C}^2_c(\R^2)$ replaced with the condition $f\in \mathcal{C}^0_c(\R^2,\R_+)$. Indeed, for any such function $f$, let $(f_k)\in \mathcal{C}^2_c(\R^2,\R_+)$ by such that $\sum_{k=1}^\infty f_k(x)=f(x)$ for all $x$. The bound then follow from applying Lemma \ref{le:tech3} on each $f_n$, since
  $\sum_k \|f_k\|_{L^\infty}=\|f\|$ and the supports of the $f_n$ are included on the support of $f$.

  Then, for $f\in \mathcal{C}^0(\R^2,\R_+)$, and $k\in \mathbb{N}\setminus\{0\}$, let $T_k=B(0,2^k)\setminus B(0,2^{k-1})$ and let $T_0=B(0,1)$. Let also $(f_k)$ be a family of positive and continuous functions, such that  $\sum f_k=f$ pointwise and such that $f_k$ is supported on $T_{k+1}\cup T_k$.
  Then,
  \[ 2^{k+1}\|f_k\|_\infty\leq 4 \sup_{z\in \R^2} (1+|z|)f_k(z). \]
  Remark also that from the disjoint supports we have \[\sum_{k\mbox{ \scriptsize odd}} \sup_{z\in \R^2} (1+|z|)f_k(z)\leq \sup_{z\in \R^2} (1+|z|)f(z) \mbox{ and }
\sum_{k\mbox{ \scriptsize even}}\sup_{z\in \R^2}  (1+|z|)f_k(z)\leq \sup_{z\in \R^2} (1+|z|)f(z).\]
  Applying Lemma \ref{le:tech3}  to each of the function $f_k$, and applying the triangle inequality, we deduce
  that
  \begin{align*}
  T^{-\frac{1}{2}-\epsilon}\mathbb{E}\Big[
  \Big(\sum_{i,j} f(T^{\frac{1}{2}}X_{iT^{-1}}, T^{\frac{1}{2}}Y_{jT^{-1}}) -T^2\int_0^1\int_0^1 f(T^{\frac{1}{2}}X_s,T^{\frac{1}{2}}X_t )\Big)^2\Big]^{\frac{1}{2}}
  &\leq 4C \sum_k \sup_{z\in \R^2} (1+|z|)f_k(z)\\
&  \leq 8C \sup_{z\in \R^2} (1+|z|)f(z).
  \end{align*}
\end{proof}

\begin{lemma}
\label{le:tech4}
  Let $p\geq 1$ and $f\in \mathcal{C}^0_b(\R^2)\cap L^p(\R^2)$. Then,
  \[ T^{\frac{1}{2}-\epsilon} \Big(\int_0^1\int_0^1 T f(T^{\frac{1}{2}}(X_s-Y_t)) \d s\d t- \ell^{X,Y}(\R^2) \int_{\R^2} f(x)\d x \Big)\overset{L^p}{\underset{t\to \infty}\longrightarrow }0.\]
\end{lemma}
\begin{proof}
  We use here the definition of the local time $\ell_y$ given in \cite[(1-a)]{LeGall}.

  From this definition, we have $\int_0^1\int_0^1 T f(T^{\frac{1}{2}})(X_s-Y_t))\d s\d t=\int_{\R^2} T f(T^{\frac{1}{2}} y) \beta(y,\mathbb{R}^2) \d y$, with
  $\beta$ that satisfies $\beta(x,\mathbb{R}^2)\underset{x\to 0}\longrightarrow \beta(0,\mathbb{R}^2)=\ell^{X,Y}(\mathbb{R}^2)$.
  In the same paper (numerotation 1-b), the following property is also given: for all $p\geq 1$ and $\epsilon>0$, there exists $C$ such that for all $x$ and $y$ in $\R^2$ and $B$ Borelian,
  \[
  \mathbb{E}[(\beta(y,B)-\beta(x,B))^p]\leq C |x-y|^{p-\epsilon}.\]
  Therefore, we have
  \begin{align*}
  \mathbb{E}\Big[ \Big|\int_0^1\int_0^1 T f(T^{\frac{1}{2}}(X_s-Y_t))\d s\d t&- \ell^{X,Y}(\R^2) \int_{\R^2}f(z) \d z\Big|^p \Big]\\
  &=\mathbb{E}\Big[ \Big| \int_{\R^2} T f(T^{\frac{1}{2}} y) \beta(y,\mathbb{R}^2) \d y -
  \int_{\R^2}f(z)\beta(0,\R^2)   \d z \Big|^p \Big]\\
  &=\mathbb{E}\Big[ \Big| \int_{B(0,K)} f(z) (\beta( T^{-\frac{1}{2}}z,\R^2)-\beta(0,\R^2) )  \d z \Big|^p \Big]\\
  &\leq \pi^{p-1} K^{2p-2} \int_{B(0,K)} |f(z)|^p \mathbb{E}\Big[ (\beta( T^{-\frac{1}{2}}z,\R^2)-\beta(0,\R^2) )^p  \Big] \d z \\
  &\leq \pi^{p-1} K^{3p-2-\epsilon}  T^{-\frac{p}{2}+\frac{\epsilon}{2} } \int_{B(0,K)} |f(z)|^p \d z.
  \end{align*}
\end{proof}

By combining together Lemma \ref{le:tech1}, Lemma \ref{le:tech2}, Corollary \ref{coro:tech3} and Lemma \ref{le:tech4}, and computing $\int_{\R^2} L_y \d y=\frac{1}{4\pi^2}$, we obtain the following proposition.
\begin{proposition}
\label{prop:asymp}
  For all $c,\epsilon>0$, there exist a constant $C$ such that for all $T\leq m\leq n\leq T^{c^{-1}}$,
  \[
  \sup_{x,y\in \R^2} \mathbb{E}\Big[ \Big(\Sigma_{n,m,T}- \frac{\ell^{i,j}(\R^2)}{4\pi^2}\Big)^2\Big]\leq C(T^{-1+\epsilon}+ T^4m^{-1}). \]
\end{proposition}

\subsection{Proof of the $L^2$ convergence in Theorem \ref{th:intersection}}

The previous proposition, applied to $n^{\pm}=n\pm T(\sqrt{n}+1)$ and $m^{\pm}=m\pm T(\sqrt{m}+1)$, gives
\[
\sup_{x,y\in \R^2} \mathbb{E}\Big[ \Big(nm \sum_{i,j=1}^T \mathcal{D}^{i,j}_{n+T(\sqrt{n}+1), m+T(\sqrt{m}+1)}- \frac{\ell^{i,j}(\R^2)}{4\pi^2}\Big)^2\Big]\leq C'(T^{-1+\epsilon}+ T^4m^{-1})\]
and
\[
\sup_{x,y\in \R^2} \mathbb{E}\Big[ \Big(nm \sum_{i,j=1}^T \mathcal{D}^{i,j}_{n-T(\sqrt{n}+1), m-T(\sqrt{m}+1)}- \frac{\ell^{i,j}(\R^2)}{4\pi^2}\Big)^2\Big]\leq C'(T^{-1+\epsilon}+ T^4m^{-1}).\]
With $T=m^{\frac{1}{5}}$, and with Proposition \ref{prop:encadrement}, we obtain:

  For all $c,\epsilon>0$, there exist a constant $C$ such that for all $ m\leq n\leq m^{c^{-1}}$,
\[
\sup_{x,y\in \R^2} \mathbb{E}\Big[ \Big(nm D^{X,Y}_{n,m} - \frac{\ell^{i,j}(\R^2)}{4\pi^2}\Big)^2\Big]\leq C(T^{-1+\epsilon}+ T^4m^{-1})\leq 2 Cm^{-\frac{1}{5}+\epsilon}.\]
This concludes the proof of the $L^2$ convergence in Theorem \ref{th:intersection}.

\section{Implementing the bootstrap}
\label{sec:boot}
The idea of the bootstrap is that the decomposition
\[ D^{X,Y}_{n,m}=\sum_{i,j=1}^T D^{i,j}_{n,m}+O(n^{-\alpha}m^{-\beta}),\]
together with the $L^2$-asymptotics
\[ D^{X,Y}_{n,m}= \frac{\ell^{X,Y}(\R^2)}{nm}+ O(n^{-1}m^{-1-\epsilon} )\]
and the asymptotic in mean
\[ \mathbb{E}[D^{X,Y}_{n,m}]= \frac{\mathbb{E}[\ell^{X,Y}(\R^2)]}{nm}+O(n^{-\gamma} m^{-\delta}), \gamma\geq\alpha, \delta\geq \beta,\]
should imply the $L^2$-asymptotics
\[ D^{X,Y}_{n,m}=\frac{\ell^{X,Y}(\R^2)}{nm}+ O(n^{-\alpha}m^{-\beta}).\]

Let us first explains shortly this idea for the set $D^X_{n}$, because this case is much simpler to deal with. Let us assume that we already know the $L^2$ bounds
\[D^X_n=\sum_{i=1}^T D^i_n+R_n, \quad R_n=O(n^{-\alpha}),\]
\[ D^{X}_{n}=\frac{1}{2\pi n}+ O(n^{-\alpha'}),\]
and
\[\mathbb{E}[D^X_n]=\frac{1}{2\pi n}+ O(n^{-\beta}), \]
with $\alpha'<\alpha\leq \beta$. Then, we can improve the middle bound by applying it to the $D^i_n$, which are i.i.d. scaled copies of $D_n^X$, $D_n^i\overset{(d)}= T^{-1}D_n$.
We then get
\[
 \Var(D^X_n)\leq 2 \sum_{i=1}^T  \Var(D^i_n)+2 \Var(R_n)
 =\frac{2}{T}\Var(D^X_n)+ O(n^{-2\alpha})= O(T^{-1} n^{-2 \alpha'}+n^{-2\alpha}),
\]
so that $D^X_n=\frac{1}{2\pi n}+O(n^{\alpha'}T^{-\frac{1}{2}}+ n^{-\alpha}+n^{-\beta} )$: the worst term $n^{-\alpha'}$ has been improved.\footnote{In this particular case, one can simply put the expression
$\frac{2}{T}\Var(D^X_n)$ on the other side of the equation, and directly obtain the bound we are searching for without any recursion, but this is a special case.}

There are at least two additional difficulties that arises in our situation. First, the extra factor $T$ that we gain during the bootstrap came from the independence between the different pieces $D^{i}_n$, $D^j_n$. On the opposite, we do not have independence between $D^{i,j}_{n,m}$ and  $D^{k,l}_{n,m}$, and we will have to show that the contribution from the covariances becomes small as $|i-k|$ and $|j-l|$ becomes large. Besides, the scaling relation is much less trivial in our situation: if we consider for example
\[ \sup_{x,y} \mathbb{E}[D^{1,1}_{n,m}D^{1,1}_{n,m} ],\]
this only scales as $T^{-2}$ instead of $T^{-4}$, so that this extra scaling factor does not compensate the $T^4$ factor coming from the sum over the indices $i,j,k,l$. There is indeed an extra $T^{-2}$ factors coming from the fact that
$D^{i,j}_{n,m}D^{k,l}_{n,m} $ is vanishing as soon as $|X_{iT^{-1}}-Y_{jT^{-1}}|$ or $|X_{kT^{-1}}-Y_{lT^{-1}}|$ is large, but this is less simple to take into account.

\subsection{Limitation of the noise}
\begin{lemma}
  \label{le:boot1}
  Let $\alpha,\beta>0$ and assume that
  \[
  \sup_{x,y} \Big(\mathbb{E}_{x,y}[(D^{X,Y}_{n,m})^2]-\mathbb{E}_{x,y}[\mathbb{E}[D^{X,Y}_{n,m}|\ell^{X,Y}(\R^2) ]^2] \Big) =O(n^{-\alpha}m^{-\beta} ).
  \]
  Then, for all $\epsilon>0$,
  \[
  \sup_{x,y}
   \mathbb{E}_{x,y}\Big[\Big(\sum_{i,j=1}^T  (D^{i,j}_{n,m}-\mathbb{E}[D^{i,j}_{n,m}|\ell^{i,j}(\R^2), X_{(i-1)T^{-1}}, Y_{(j-1)T^{-1}} ])     \Big)^2 \Big] =O ( T^{-1+\epsilon} n^{-\alpha}m^{-\beta} ).
  \]
\end{lemma}
\begin{proof}
  %
  %
  The proof mimics the proof of Lemma \ref{le:tech1}.
  We set \[Q^{i,j}_{n,m}=D^{i,j}_{n,m}-\mathbb{E}[D^{i,j}|\ell^{i,j}(\R^2), X_{(i-1)T^{-1}}, X_{iT^{-1}} , Y_{(j-1)T^{-1}}   ].\]

  \diams For $i<k$ and $j<l$, $\mathbb{E}[Q^{i,j}_{n,m}Q^{k,l}_{n,m}]=0$ clearly.

  \diams For $i<k$ and $l<j$, we set \[\sigma^{i,j}=\sigma(  X_{(i-1)T^{-1}}, X_{iT^{-1}} , Y_{(j-1)T^{-1}} , (X_t)_{t\in [(k-1)T^{-1}, kT^{-1}] }, (Y_s)_{s\in [(k-1)T^{-1}, kT^{-1}] } ).\]
  Then, $\mathbb{E}[Q^{i,j}_{n,m}|\sigma^{i,j}]=0$ whilst $Q^{k,l}_{n,m}$ is $\sigma^{i,j}$-measurable. We deduce that $\mathbb{E}[Q^{i,j}_{n,m}Q^{k,l}_{n,m}]=0$.

  \diams For $i=k$ and $l=j$,
  let us set
  \[Q^{X,Y}_{n,m}=D^{X,Y}_{n,m}-\mathbb{E}_{x,y}[D^{X,Y}|\ell^{X,Y}(\R^2), X_1 ] , \quad
  \mbox{and} \quad
  R^{X,Y}_{n,m}=D^{X,Y}_{n,m}-\mathbb{E}_{x,y}[D^{X,Y}|\ell^{X,Y}(\R^2)].
  \]
  Remark that $Q^{X,Y}_{n,m}=R^{X,Y}_{n,m}-\mathbb{E}[R^{X,Y}_{n,m}|\ell^{X,Y}(\R^2), X_1 ] $, so that
  \[\|Q^{X,Y}_{n,m}\|_{L^2}^2\leq \|R^{X,Y}_{n,m}\|_{L^2}^2=O(n^{-\alpha}m^{-\beta}). \]
  We have
  \begin{align*}
  \mathbb{E}_{x,y}[&(Q^{i,j}_{n,m})^2 ]
  = \int_{(\R^2)^2}\hspace{-0.3cm} p_{(i-1)T^{-1}}(x,x_2)p_{(j-1)T^{-1}}(y,y_2)\mathbb{E}_{x,y}[(Q^{i,j}_{n,m})^2|X_{(i-1)T^{-1}}=x_2,Y_{(j-1)T^{-1}}=y_2  ]\\
  &= T^{-2} \int_{(\R^2)^2} p_{(i-1)T^{-1}}(x,x_2)p_{(j-1)T^{-1}}(y,y_2) \mathbb{E}_{\sqrt{T}x_2, \sqrt{T}y_2 }[(Q^{X,Y}_{n,m})^2]\\
  &\leq C T^{-3+4\epsilon} \frac{1}{(i-1)T^{-1}(j-1)T^{-1}} \sup_{x',y'} \mathbb{E}_{x',y'}[(Q^{X,Y}_{n,m})^2]
  + T^{-2} \hspace{-0.3cm}\sup_{x',y': |x'-y'|\geq T^{\epsilon }} \hspace{-0.3cm} \mathbb{E}_{x',y'}[(Q^{X,Y}_{n,m})^2]\\
  &\leq C' T^{-3+4\epsilon} \frac{1}{(i-1)T^{-1}(j-1)T^{-1}} n^{-\alpha} m^{-\beta}+ T^{-r},
  \end{align*}
  with $r>0$ arbitrary.

  \diams We finally look at the expression $\mathbb{E}_{x,y}[Q^{i,j}_{n,m}Q^{i,l}_{n,m}]$, with $j<l$. Once again, we know that \[\mathbb{E}_{x,y}[Q^{i,j}_{n,m}Q^{i,l}_{n,m}\mathbbm{1}_{|Y_{(j-1)T^{-1}}-X_{(i-1)T^{-1}} |\geq T^{-\frac{1}{2}+\epsilon} \mbox{ or } |Y_{(k-1)T^{-1}}-X_{(i-1)T^{-1}}|\geq T^{-\frac{1}{2}+\epsilon} } ]\leq T^{-r}.\]
  We have
  \begin{align*}
  \mathbb{E}_{x,y}[&Q^{i,j}_{n,m}Q^{i,l}_{n,m}\mathbbm{1}_{|Y_{(j-1)T^{-1}}-X_{(i-1)T^{-1}} |\leq T^{-\frac{1}{2}+\epsilon} } \mathbbm{1}_{|Y_{(k-1)T^{-1}}-X_{(i-1)T^{-1}}|\geq T^{-\frac{1}{2}+\epsilon} }   ]\\
  &\leq \mathbb{P}_{x,y}(|Y_{(j-1)T^{-1}}-X_{(i-1)T^{-1}} |\leq T^{-\frac{1}{2}+\epsilon} )
  \sup_{z}\mathbb{P}_{z}(|Y_{(l-j)T^{-1}}| \leq T^{-\frac{1}{2}+\epsilon} )\\
  &   \mathbb{E}_{x,y}[Q^{i,j}_{n,m}Q^{i,l}_{n,m}| |Y_{(j-1)T^{-1}}-X_{(i-1)T^{-1}} |\leq T^{-\frac{1}{2}+\epsilon}, |Y_{(k-1)T^{-1}}-X_{(i-1)T^{-1}}|\geq T^{-\frac{1}{2}+\epsilon}  )\\
  &\leq \frac{C T^{-2+4 \epsilon} }{ (i-1)T^{-1}(j-i)T^{-1} } \sup_{x',y'}   \mathbb{E}_{x,y}[(Q^{i,j}_{n,m})^2]\\
  &\leq \frac{C T^{-2+4 \epsilon} }{ (i-1)T^{-1}(j-i)T^{-1} } n^{-\alpha} m^{-\beta}.
  \end{align*}

  Summing over all the indices $i,j,k,l$, we obtain
  \[ \sum_{i,j,k,l} \mathbb{E}_{x,y}[Q^{i,j}_{n,m}Q^{k,l}_{n,m}]\leq C T^{-1+\epsilon'} n^{-\alpha}m^{-\beta}.\]
  We conclude as in Lemma \ref{le:tech1}.

\end{proof}
%

\subsection{Elimination of the indices $n,m$}
Now, we take a different turn than in the previous section. Trying to follow the same path as before would be doomed, because a quantity that would depend only on the $X_{iT^{-1}}, Y_{jT^{-1}}$ cannot be a sufficiently good approximation of $\ell^{X,Y}(\R^2)$ unless $T$ is large, which we don't want.
\begin{lemma}
  \label{le:boot2}
  Let $\beta>0$ , and assume that
  \[
  \sup_{x,y} \mathbb{E}_{x,y}[ (nm \mathbb{E}_{x,y}[ D^{X,Y}_{n,m}|\ell^{X,Y}(\R^2)]-\ell^{X,Y}(\R^2))^2] =O(m^{-\beta} ).
  \]
  Then, for all $\omega<\frac{1}{2}$
  \[
  \sup_{x,y}
 \mathbb{E}_{x,y}\Big[\Big( \sum_{i,j=1}^T (nm\mathbb{E}[ D^{i,j}_{n,m}|\ell^{i,j}(\R^2), X_{(i-1)T^{-1}}, Y_{(j-1)T^{-1}} ] -\ell^{i,j}(\R^2) ) \Big)^2 \Big] =O ( T^{-\omega} m^{-\beta}+T^2 m^{-\frac{1}{2}-\frac{\beta}{2}  } ).
  \]
\end{lemma}
\begin{proof}
  We set \[E^{i,j}_{n,m}=nm \mathbb{E}[ D^{i,j}_{n,m}|\ell^{i,j}(\R^2), X_{(i-1)T^{-1}}, Y_{(j-1)T^{-1}} ] -\ell^{i,j}(\R^2),\]
  and
  \[ F^{i,j}_{n,m}= E^{i,j}_{n,m} \mathbbm{1}_{|X_{iT^{-1}}-X_{(i-1)T^{-1}} |\leq T^{-\frac{1}{2}+\epsilon}  }
  \mathbbm{1}_{|Y_{jT^{-1}}-Y_{(j-1)T^{-1}} |\leq T^{-\frac{1}{2}+\epsilon}  }
  \mathbbm{1}_{|X_{(i-1)T^{-1}}-Y_{(j-1)T^{-1} }|\leq T^{-\frac{1}{2}+\epsilon }}.
  \]
  In this proof, $\sup_{y'_1}$ always means $\sup_{y'_1\in B(y_1,T^{-\frac{1}{2}+\epsilon})}$.

  \diams As usual, the case $i<k-1$, $j<l-1$ is easily dealt with. Let $Z=( X_{ iT^{-1}}, Y_{(j-1) T^{-1}}) $.
  Then,
  \begin{align*}
  |\mathbb{E}_{x,y}\big[F^{k,l}_{n,m}\big|Z\big]|&\leq \mathbb{P}\big(|X_{(k-1)T^{-1}}-Y_{(l-1)T^{-1}}|\leq T^{-\frac{1}{2}+\epsilon}\big|Z\big)   \\
  &\hspace{5cm}\sup_{x',y'} |\mathbb{E}_{x,y}\big[F^{k,l}_{n,m}\big|X_{(k-1)T^{-1}}=x', Y_{(l-1)T^{-1}}=y' \big]|\\
  &\leq \frac{C T^{-1+2\epsilon}}{(k+l-i-j-2)T^{-1} } T^{-1} \sup_{x',y'}|\mathbb{E}_{x',y'}\big[F^{X,Y}_{n,m}]|\\
  &\leq \frac{C' T^{-2+2\epsilon}}{(k+l-i-j-2)T^{-1} }  m^{-\frac{1}{2}} \qquad \mbox{using Corollary \ref{coro:werner}}.
  \end{align*}
  It follows that
  \begin{align*}
  \mathbb{E}_{x,y}[ F^{i,j}_{n,m} F^{k,l}_{n,m}  ]
  &= \mathbb{E}_{x,y}\big[ F^{i,j}_{n,m} \mathbb{E}_{x,y}\big[F^{k,l}_{n,m}\big|Z\big]  \big]\\
  & \leq  \frac{C' T^{-2+2\epsilon}}{(k+l-i-j-2)T^{-1} }  m^{-\frac{1}{2}}\mathbb{E}_{x,y}[ |F^{i,j}_{n,m}| ]\\
  & \leq  \frac{C' T^{-2+2\epsilon}}{(k+l-i-j-2)T^{-1} }  m^{-\frac{1}{2}}T^{-1}\mathbb{P}(|X_{(i-1)T^{-1}}-Y_{(j-1)T^{-1} }|\leq T^{-\frac{1}{2}+\epsilon } )\\ &\hspace{8.5cm}\sup_{x',y'} \mathbb{E}_{x',y'}[ (F^{i,j}_{n,m})^2 ]^{\frac{1}{2}}\\
  &\leq  \frac{C'' T^{-4+4\epsilon}}{(k+l-i-j-2)(i-j-2)T^{-2} }  m^{-\frac{1}{2}-\frac{\beta}{2}}.
  \end{align*}
  The sum over $i,j,k,l$ is then smaller than
  \[C T^{4\epsilon} m^{-\frac{1}{2}-\frac{\beta}{2}}.\]
  Remark that we cannot get a factor $m^{-1}$ because of the absolute value around $F^{i,j}_{n,m}$, and that our bound would not have been sufficient if we would also have absolute values around $F^{k,l}_{n,m}$: keeping one of this term with no absolute value is the requirement that makes the next case difficult to deal with.

  \diams We now assume $i<k-1$, $j>l$, but for once we also assume $j-l>T^{\delta}$, for some fixed but arbitrary $\delta>0$.

  Let $Z=( X_{ iT^{-1}}, Y_{(j-1) T^{-1}}) $, $W=(Z, Y_{(l-1) T^{-1}})$ and $\mathcal{F}=\sigma(Z, (X_t)_{t<iT^{-1}}, (Y_s)_{s> (j-1)T^{-1}}) $.
  Remark that
  \begin{align*} \mathbb{E}_{x,y}[ F^{i,j}_{n,m} F^{k,l}_{n,m}  ]&= \mathbb{E}_{x,y}[  F^{i,j}_{n,m} \mathbb{E}_{x,y}[ F^{k,l}_{n,m}|\mathcal{F}  ]]=
  \mathbb{E}_{x,y}[  F^{i,j}_{n,m} \mathbb{E}_{x,y}[ F^{k,l}_{n,m}|Z  ]]\\
  &=
  \mathbb{E}_{x,y}[  \mathbb{E}_{x,y}[  F^{i,j}_{n,m}|Z] \mathbb{E}_{x,y}[ F^{k,l}_{n,m}|Z  ]].
  \end{align*}
  Disintegrating $  \mathbb{E}_{x,y}[ F^{i,j}_{n,m} F^{k,l}_{n,m}  ]$ with respect to $Z=( X_{ iT^{-1}}, Y_{(j-1) T^{-1}} )$, and disintegrating then
  $\mathbb{E}_{x,y}[ F^{k,l}_{n,m}|  Z=(x_1,y_2)]$ with respect to $Y_{(l-1)T^{-1}}$, we get
  \begin{align*}
  \mathbb{E}_{x,y}[ F^{i,j}_{n,m} F^{k,l}_{n,m}  ] &=\int_{(\R^2)^2} p_{i T^{-1}}(x,x_1)  p_{(j-1) T^{-1}}(y,y_2)
  \\
  & \hspace{3cm} \mathbb{E}_{x,y}[F^{i,j}_{n,m} |  Z=(x_1,y_2)] \mathbb{E}_{x,y}[ F^{k,l}_{n,m}|  Z=(x_1,y_2)]
  \d x_1  \d y_2,\\
  &=\int_{H} p_{i T^{-1}}(x,x_1)  p_{(l-1) T^{-1}}(y,y_1)  p_{(j-l) T^{-1}}(y_1,y_2)
  \\
  & \hspace{1cm} \mathbb{E}_{x,y}[F^{i,j}_{n,m} |  Z=(x_1,y_2)] \mathbb{E}_{x,y}[ F^{k,l}_{n,m}|  W=(x_1,y_2,y_1)]
  \d x_1  \d y_1 \d y_2,
  \end{align*}
  where $H\subseteq (\R^2)^3$ is the set of triples $(x_1,y_1,y_2)$ with $|y_2-x_1|\leq T^{-\frac{1}{2}+\epsilon}$.
  Under \[\mathbb{P}_{x,y}(\ \cdot \ | W= (x_1,y_2,y_1)),\] $(X_s)_{s\geq {iT^{-1}}}$ is a Brownian motion started from $x_1$, and $(Y_t)_{t\in [ (l-1)T^{-1}, (j-1)T^{-1}]}$ is a Brownian bridge from $y_1$ to $y_2$ of duration $j-l$. When further restricted to
  $t\in [ (l-1)T^{-1}, lT^{-1}]$,
  its distribution becomes absolutely continuous with respect to the one of the Brownian motion started from $y_1$, with Radon-Nikodym derivative given by
  \[
  \Big(\frac{\mathbb{P}_{y}(\ \cdot \ | W= (x_1,y_2,y_1))  }{ \mathbb{P}_{y_1} }\Big)_{\big|\sigma( (Y_t)_{t \in [ (l-1)T^{-1}, lT^{-1}] } }=  \frac{p_{(j-l-1)T^{-1}    }(Y_{lT^{-1}},y_2)   }{p_{(j-l)T^{-1} }(y_1,y_2) }.
  \]
  For any two measures $\mathbb{Q}\ll \mathbb{P}$ and variable $X$ (not necessarily positive),
  \[
  |\mathbb{E}_{\mathbb{Q}}[X]-\mathbb{E}_{\mathbb{P}}[X]|=\Big|\int X(\omega) \big(\frac{\d \mathbb{Q}}{\d \mathbb{P}}(\omega) -1\Big) \d \mathbb{P}\leq \mathbb{E}_{\mathbb{P}}[X^2]^{\frac{1}{2}}\mathbb{E}_{\mathbb{P}}\Big[ \Big(\frac{\d \mathbb{Q}}{\d \mathbb{P}} -1  \Big)^2 \Big]^{\frac{1}{2}}.
  \]
  In our case, the second factor reads
  \[ A(y_1,y_2)= \int_{\R^2}  \Big( \frac{p_{(j-l-1)T^{-1}    }(y ,y_2)   }{p_{(j-l)T^{-1} }(y_1,y_2) } -1\Big)^2 p_{T^{-1}}(y_1,y)\d y.\]
  The most simple way to compute this integral is to develop the square, and to us the fact that $p_t(x,y)^2=\frac{1}{2\pi t} p_{t/2}(x,y)$.
  All computations done, we get
  \[ A(y_1,y_2)= \frac{j-l-1}{j-l+1}p_{(j-l-1)T^{-1}}(y_1,y_2)-p_{(j-l)T^{-1}}(y_1,y_2).\]
  We deduce
  \begin{align*}
  \mathbb{E}_{x,y}[ F^{i,j}_{n,m} F^{k,l}_{n,m}  ]
  &\leq 
  \int_{H} p_{i T^{-1}}(x,x_1)  p_{(l-1) T^{-1}}(y,y_1)  p_{(j-l) T^{-1}}(y_1,y_2)\mathbb{E}_{x,y}[F^{i,j}_{n,m} |  Z=(x_1,y_2)]
  \\
  & \hspace{5cm} \mathbb{E}_{x,y}[ F^{k,l}_{n,m}| X_{iT^{-1}}=x_1 , Y_{(l-1)T^{-1}}=y_1]
  \d x_1  \d y_1 \d y_2\\
  &
  +  \int_{H} p_{i T^{-1}}(x,x_1)  p_{(l-1) T^{-1}}(y,y_1) A(y_1,y_2)
  \\
  & \hspace{0.2cm} |\mathbb{E}_{x,y}[F^{i,j}_{n,m} |  Z=(x_1,y_2)]| \mathbb{E}_{x,y}\big[ (F^{k,l}_{n,m})^2 \big|X_{iT^{-1}}=x_1, Y_{(l-1)T^{-1}}=y_1\big]^{\frac{1}{2}} \d x_1  \d y_1 \d y_2\\
  &=I_1+I_2.
  \end{align*}
  For the first term $I_1$, we can integrate back with respect to $y_2$, so that
  \begin{align*}
  I_1=&\int_{(\R^2)^2} p_{i T^{-1}}(x,x_1)  p_{(l-1) T^{-1}}(y,y_1)  \mathbb{E}_{x,y}[F^{i,j}_{n,m} |X_{iT^{-1}}=x_1, Y_{(l-1)T^{-1}}=y_1 ]
  \\
  & \hspace{1cm} \mathbb{E}_{x,y}[ F^{k,l}_{n,m}| X_{iT^{-1}}=x_1 , Y_{(l-1)T^{-1}}=y_1]
  \d x_1  \d y_1 \\
  &= \mathbb{E}_{x,y}[ \mathbb{E}[F^{i,j}_{n,m}|  X_{iT^{-1}} , Y_{(l-1)T^{-1}}] \mathbb{E}[F^{k,l}_{n,m}|  X_{iT^{-1}} , Y_{(l-1)T^{-1}} ]]
  \end{align*}
  Now,
  \begin{align*}
  |\mathbb{E}\big[F^{k,l}_{n,m}\big|  X_{iT^{-1}} , Y_{(l-1)T^{-1}} \big]|
  &\leq \mathbb{E}\big[|\mathbb{E}\big[ F^{k,l}_{n,m}\big|  X_{(k-1)T^{-1}} , Y_{(l-1)T^{-1}} \big]| \big|  X_{iT^{-1}} , Y_{(l-1)T^{-1}} \big]\\
  &\leq CT^{-1}m^{-\frac{1}{2}}+ \mathbb{E}[|E^{k,l}_{n,m}-F^{k,l}_{n,m}|\big|  X_{iT^{-1}} , Y_{(l-1)T^{-1}} \big]|.
  \end{align*}
  Proceeding as we already did several times, we easily show that $\mathbb{E}[|E^{k,l}_{n,m}-F^{k,l}_{n,m}|^2]$ decays to $0$ more quickly than any power of $T$,
  and we deduce that
  \[
  I_1\leq (CT^{-1}m^{-\frac{1}{2}}+\mathbb{E}[|E^{k,l}_{n,m}-F^{k,l}_{n,m}|^2]^\frac{1}{2}) \mathbb{E}[(E^{i,j}_{n,m})^2 ]^{\frac{1}{2}}\leq C'T^{-2}m^{-\frac{1}{2}}m^{-\frac{\beta}{2}}.
  \]

  For the second term $I_2$, we split $H$ into $H'\sqcup H''$ with \[H'=\{(x_1,y_1,y_2)\in H: |y_1-y_2|\leq (j-l)^\frac{1}{2} T^{-\frac{1}{2}+\frac{\delta}{2}}\},\] and we decompose $I_2$ into $I_2'+I_2''$ accordingly. We first have
  \begin{align}
  \mathbb{E}_{x,y}\big[ (F^{k,l}_{n,m})^2 \big|X_{iT^{-1}}=x_1, Y_{(l-1)T^{-1}}=y_1\big]^{\frac{1}{2}}
  &\leq T^{-1} \mathbb{P}(|X_{(k-1)T^{-1}}-Y_{(j-1)T^{-1}}|\leq T^{-\frac{1}{2}+\epsilon} )^{\frac{1}{2}}\nonumber\\
&\hspace{4cm}  \sup_{x',y'} \mathbb{E}_{x',y'}\big[ (E^{X,Y}_{n,m})^2\big]^{\frac{1}{2}}\nonumber\\
  &\leq C \frac{T^{-\frac{3}{2}+\epsilon}}{(k-j)^{\frac{1}{2}}T^{-\frac{1}{2}} } m^{-\frac{\beta}{2}}.\label{eq:Fkl}
  \end{align}

  On $H'$, $A(y_1,y_2)$ is less than
  \begin{equation}
  \frac{(j-l)^2}{(j-l)^2-1}\exp( \frac{T^{\delta}}{j-l} )-1\leq C \frac{T^{\delta}}{j-l}.
  \label{eq:computA}
  \end{equation}

  Using \eqref{eq:computA} and \eqref{eq:Fkl}, we obtain
  \begin{align*}
  I'_2
  &\leq C\frac{T^{-\frac{3}{2}+\epsilon}}{(k-j)^{\frac{1}{2}}T^{-\frac{1}{2}} } m^{-\frac{\beta}{2}}\frac{T^{\delta}}{j-l} \int_{H} p_{i T^{-1}}(x,x_1)  p_{(l-1) T^{-1}}(y,y_1)  \\
  &\hspace{4cm} \mathbb{E}_{x,y}\big[ |F^{i,j}_{n,m}| \big|X_{iT^{-1}}=x_1, Y_{(l-1)T^{-1}}=y_1\big] \d x_1  \d y_1 \d y_2\\
  &\leq C'\frac{T^{-\frac{3}{2}+\epsilon}}{(k-j)^{\frac{1}{2}}T^{-\frac{1}{2}} } m^{-\frac{\beta}{2}} \frac{T^{\delta}}{j-l}  (j-l)   T^{-1+\delta}
  \int_{(\R^2)^2} p_{i T^{-1}}(x,x_1)  p_{(l-1) T^{-1}}(y,y_1)\\
  &\hspace{4cm} \mathbb{E}_{x,y}\big[ |F^{i,j}_{n,m}| \big|X_{iT^{-1}}=x_1, Y_{(l-1)T^{-1}}=y_1\big] \d x_1  \d y_1 \\
  &=C'\frac{T^{-\frac{5}{2}+\epsilon+2\delta}}{(k-j)^{\frac{1}{2}}T^{-\frac{1}{2}} } m^{-\frac{\beta}{2}}    \mathbb{E}_{x,y}\big[ |F^{i,j}_{n,m}| \big]\\
  &\leq
  C''\frac{T^{-\frac{5}{2}+\epsilon+2\delta}}{(k-j)^{\frac{1}{2}}T^{-\frac{1}{2}} } m^{-\frac{\beta}{2}}   \frac{T^{-2+2\epsilon}}{ (i+j)T^{-1}}  \sup_{x',y'}\mathbb{E}_{x',y'}\big[ (F^{X,Y}_{n,m})^2 \big]^{\frac{1}{2}}\\
  &\leq
  C^{(3)}\frac{T^{-\frac{9}{2}+3\epsilon+2\delta}}{(k-j)^{\frac{1}{2}}T^{-\frac{1}{2}} (i+j)T^{-1} } m^{-\beta}  .
  \end{align*}

  On $H''$, we simply bound $A(y_1,y_2)$ by
  \[ p_{(j-l-1)T^{-1}}(y_1,y_2)+p_{(j-l)T^{-1}}(y_1,y_2).\]

  For any $\delta>0$, the integral over $y_2\in \R^2\setminus B(y_1,(j-l)^{\frac{1}{2}}T^{-\frac{1}{2}+\delta})$ decays more quickly than any power of $T$, and we get
  \begin{align*}
  I''_2
  &\leq C \frac{T^{-2+2\epsilon}}{(k-j)^{\frac{1}{2}}T^{-\frac{1}{2}} } m^{-\frac{\beta}{2}}   \int_{H'} p_{i T^{-1}}(x,x_1)  p_{(l-1) T^{-1}}(y,y_1) \\
  &\hspace{6cm}(p_{(j-l) T^{-1}}(y_1,y_2)+ p_{(j-l-1) T^{-1}}( y_1-T^{-\frac{1}{2} +\epsilon}(y_1-y_2) ,y_2)) \\
  &\hspace{6cm}\mathbb{E}_{x,y}\big[ |F^{i,j}_{n,m}| \big|X_{iT^{-1}}=x_1, Y_{(l-1)T^{-1}}=y_1\big] \d x_1  \d y_1 \d y_2\\
  &\leq C T^{-r}.
  \end{align*}
  All together, and using again that $\mathbb{E}_{x,y}[(E^{i,j}_{n,m}-F^{i,j}_{n,m})^2]$ decays more quickly than any power of $T$, we have
  \[ \mathbb{E}_{x,y}[E^{i,j}_{n,m}E^{k,l}_{n,m}]\leq \mathbb{E}_{x,y}[F^{i,j}_{n,m}F^{k,l}_{n,m}]+O(T^{-r})\leq
  +C T^{-2} m^{-\frac{1}{2}-\frac{\beta}{2}}
  +  C \frac{T^{-\frac{9}{2}+3\epsilon+2\delta}}{(k-j)^{\frac{1}{2}}T^{-\frac{1}{2}} (i+j)T^{-1} } m^{-\beta}
  .\]
  The sum over all $i<k,j>l+T^\delta$ is less than
  \[ C'(T^{2} m^{-\frac{1}{2}-\frac{\beta}{2}}  +  T^{-\frac{1}{2}+3\epsilon+2\delta}m^{-\beta}).
  \]

  \diams Of course, we can interchange the roles of $(i,j)$ and $(l,k)$, so we can bound identically the sum over $i+T^{\delta}<k$, $j>l$, so that there only remains the about $T^{2+2\delta}$ quadruples with $|j-l|<T^{\delta}$, $|k-i|<T^{\delta}$. We now deal with them.

  First, disintegrating with respect to $U=(X_{(i-1)T^{-1}},Y_{(j-1)T^{-1}})$,
  we have
  \begin{align*}
  \mathbb{E}_{x,y}[(E^{i,j}_{n,m})^2]
  &\leq \mathbb{P}_{x,y}(|X_{(i-1)T^{-1}}-Y_{(j-1)T^{-1}}|\leq T^{-\frac{1}{2}+\epsilon } )\sup_{x',y'}\mathbb{E}_{x,y}[(E^{i,j}_{n,m})^2| U=(x',y')]+O(T^{-r})\\
  &\leq C\frac{T^{-1+2\epsilon}}{ (i-1)T^{-1}(j-1)T^{-1}} T^{-2} \sup_{x',y'}\mathbb{E}_{x',y'}[(E^{X,Y}_{n,m})^2]\\
  &\leq C' \frac{T^{-\frac{5}{2}+2\epsilon}}{ (i-1)T^{-1}(j-1)T^{-1}} m^{-\beta}.
  \end{align*}
  From Cauchy--Schwarz inequality, we deduce the same bound for
  $\mathbb{E}_{x,y}[E^{i,j}_{n,m}E^{k,l}_{n,m}]$. The sum over the quadruples $(i,j,k,l)$ with $|j-l|<T^{\delta}$, $|k-i|<T^{\delta}$ is therefore less than
  \[ C'' \log(T)^2 T^{-\frac{1}{2}+2\epsilon+2\delta} m^{-\beta}.\]

  \diams It remains to deal with the terms $\mathbb{E}_{x,y}[E^{i,j}_{n,m}E^{i,l}_{n,m}]$. We assume $l<j$.
  Set \[ U=(X_{(i-1)T^{-1}},Y_{(j-1)T^{-1}}),\mbox{ and }Z=Y_{(l-1)T^{-1}}-Y_{jT^{-1}}.\] Remark that for $E^{i,j}_{n,m}E^{i,l}_{n,m}$ to be nonzero, it requires both $|X_{(i-1)T^{-1}}-Y_{(j-1)T^{-1}}|$
  and $|X_{(i-1)T^{-1}}-Y_{(l-1)T^{-1}}|$ to be small, in which case $Z$ is small as well.
  Then, disintegrating with respect to $U$ and $Z$, we get
  \begin{align*}
  \mathbb{E}_{x,y}[E^{i,j}_{n,m}E^{i,l}_{n,m}]
  &\leq \mathbb{P}_{x,y}(|X_{(i-1)T^{-1}}-Y_{(j-1)T^{-1}}|\leq T^{-\frac{1}{2}+\epsilon } )
  \mathbb{P}_{x,y}(|Z|\leq T^{-\frac{1}{2}+\epsilon } )\\
&\hspace{4cm}  \sup_{x',y'}\mathbb{E}_{x,y}[E^{i,j}_{n,m}E^{i,l}_{n,m}| U=(x',y'), Z=z]+O(T^{-r})\\
  &\leq C\frac{T^{-2+2\epsilon }}{(i+j)T^{-1}(l-j)T^{-1}   } T^{-2} \sup_{x',y'}\mathbb{E}_{x',y'}[(E^{X,Y}_{n,m})^2]+O(T^{-r})\\
  &\leq C\frac{T^{-4+2\epsilon }}{(i+j)T^{-1}(l-j)T^{-1}   }  m^{-\beta}.
  \end{align*}
  The sum of these terms over $i,j,l$ is less than
  \[
  C' \log(T) T^{-1+2\epsilon} m^{-\beta},
  \]
  which concludes the proof.
\end{proof}

\subsection{End of the bootstrap}
From the two previous lemma \ref{le:boot1} and \ref{le:boot2}, we easily deduce the following.
\begin{corollary}
\label{coro:boot}
  Let $\beta>0$ and assume that
  \[
  \sup_{x,y} \mathbb{E}_{x,y}\Big[ \Big( nm D^{X,Y}_{n,m}-\ell^{X,Y}(\R^2)\Big)^2\Big]\leq O(m^{-\beta}).
  \]
  Then, for all $\epsilon>0$,
  \[
  \sup_{x,y} \mathbb{E}_{x,y}\Big[ \Big( nm D^{X,Y}_{n,m}-\ell^{X,Y}(\R^2)\Big)^2\Big] \leq O(T^{-\frac{1}{2}+\epsilon} m^{-\beta}+T^2m^{-\frac{1}{2}-\frac{\beta}{2}} +\log(T)^2T^6 m^{-1} ).
  \]
\end{corollary}
\begin{proof}
  Since
  \[
  \mathbb{E}_{x,y}\Big[ \mathbb{E}_{x,y}\big[ nm D^{X,Y}_{n,m} -\ell^{X,Y}(\R^2)| \ell^{X,Y}(\R^2) \big]^2\Big]\leq
  \mathbb{E}_{x,y}\Big[ \Big( nm D^{X,Y}_{n,m}-\ell^{X,Y}(\R^2)\Big)^2\Big]\leq O(m^{-\beta}),\]
  the assumption of Lemma \ref{le:boot2} is satisfied.

  Since $\mathbb{E}[(X-Y)^2]\geq\mathbb{E}[(X-\mathbb{E}[X|\sigma])^2]=\mathbb{E}[X^2]-\mathbb{E}[\mathbb{E}[X|\sigma])^2] $ for any $\sigma$-measurable variable $Y$, the assumption of Lemma \ref{le:boot2} is also satisfied\footnote{with $\alpha=1$ and a shift of $1$ in the exponent $\beta$
  between Lemma \ref{le:boot2} and here}.

  Thus,
  \begin{align*}
  \sup_{x,y} \mathbb{E}_{x,y}\Big[ \Big( \sum_{i,j=1}^T (nm D^{i,j}_{n,m}-\ell^{i,j}(\R^2))\Big)^2\Big]
  &\leq 2 \sup_{x,y} \mathbb{E}_{x,y}\Big[ \Big( \sum_{i,j=1}^T nm( D^{i,j}_{n,m}-\mathbb{E}[D^{i,j}_{n,m} |\ell^{i,j}(\R^2)] )\Big)^2\Big]\\
  &+2\sup_{x,y} \mathbb{E}_{x,y}\Big[ \Big( \sum_{i,j=1}^T( nm\mathbb{E}[D^{i,j}_{n,m} |\ell^{i,j}(\R^2)]-\ell^{i,j}(\R^2) )\Big)^2\Big]\\
  &=O(T^{-\frac{1}{2}+\epsilon} m^{-\beta}+T^2m^{-\frac{1}{2}-\frac{\beta}{2}} ).
  \end{align*}
  Using now the fact that $\sum \ell^{i,j}(\R^2)=\ell^{X,Y}$ and Proposition \ref{prop:encadrement}, setting $n^\pm=n\pm T(\sqrt{n}+1)$ and $m^\pm=m\pm T(\sqrt{m}+1)$, and denoting by $x_{\pm}$ the positive (resp. negative) part of $x$, we obtain
  we obtain
  \begin{align*}
  \sup_{x,y} \mathbb{E}_{x,y}\Big[ \Big( nm D^{X,Y}_{n,m}&-\ell^{X,Y}(\R^2)\Big)_+^2\Big]\leq
  \sup_{x,y} \mathbb{E}_{x,y}\Big[ \Big( \sum_{i,j=1}^T (nm D^{i,j}_{n^-,m^-}\hspace{-0.5cm}-\ell^{i,j}(\R^2))\Big)^2\Big]+O(\log(T)^2T^6 m^{-1} )\\
  &\leq
  \sup_{x,y} \mathbb{E}_{x,y}\Big[ \Big( \sum_{i,j=1}^T (n^-m^- D^{i,j}_{n^-,m^-}-\ell^{i,j}(\R^2))\Big)^2\Big]+O(\log(T)^2T^6 m^{-1} )\\
  &\leq O(T^{-\frac{1}{2}+\epsilon} m^{-\beta}+T^2m^{-\frac{1}{2}-\frac{\beta}{2}}+\log(T)^2T^6 m^{-1} ).
  \end{align*}
  The negative part is treated identically, and this concludes the proof.
\end{proof}
Finally, we conclude that
\begin{proposition}
  For all $\epsilon>0$,
  \[
  \sup_{x,y} \mathbb{E}_{x,y}\Big[ \Big( nm D^{X,Y}_{n,m}-\ell^{X,Y}(\R^2)\Big)^2\Big]\leq O(m^{-1+\epsilon}).
  \]
\end{proposition}
\begin{proof}
  Let $\beta_0$ be the supremum of the values $\beta$ such that
    \[
    \sup_{x,y} \mathbb{E}_{x,y}\Big[ \Big( nm D^{X,Y}_{n,m}-\ell^{X,Y}(\R^2)\Big)^2\Big]\leq O(m^{-\beta}).
    \]
  If $\beta_0<1$, let $\beta\in \big(\beta_0-\tfrac{1-\beta_0}{12}, \beta_0)$ and let $T$ be equivalent to $n^\omega$ for $\omega$ in the non-empty interval $\big(2\beta_0-2\beta, \frac{1-\beta_0}{12}\big)$. Then, Corollary \ref{coro:boot} gives a bound better than $\beta_0$, which is absurd.
  Therefore, $\beta_0\geq 1$.
\end{proof}

\section{From $L^2$ to $L^p$ and a.s.}
\label{sec:lp}
Our goal in this section is to extend the previous estimations from $L^2$ to $L^p$. Our strategy is not to start over the whole proof, but instead to control the large deviations of $D^{X,Y}_{n,m}$ around $\ell^{X,Y}$. The main estimation we need is the following.
\begin{lemma}
  \label{le:concentration}
  For any $r>0$ and $\epsilon, \epsilon'>0$, there is a constant $C$ such that for all $n,m$ with $n^{\epsilon'}\leq m\leq n$,
  \[
  \sup_{x,y}\mathbb{P}_{x,y}\big(\big|nmD_{n,m}^{X,Y}-\frac{\ell^{X,Y}}{4\pi^2}\big|\geq m^{-\frac{1}{2}+\epsilon} \big)\leq Cm^{-r}.
  \]
\end{lemma}
  \begin{proof}[Sketch of the proof]
  Let $p_\epsilon$ be the left hand side.
  We first remark that
  \begin{itemize}
  \item For $\epsilon>\frac{1}{2}$,
  $p_\epsilon$ decays more quickly than any power of $m$.
  \item For all $\epsilon>0$, there exists $r_\epsilon>0$ such that $p_\epsilon$ decays more quickly than $m^{-r_\epsilon}$.
  \end{itemize}
  We are going to recursively improve this $r_\epsilon$, by showing an inequality which is roughly
  \[ p_\epsilon \leq T^2 p_{\epsilon+ \omega \epsilon' }+ T^4 p_{\epsilon- \omega\epsilon''}^2,\]
  for some $\epsilon', \epsilon'',\omega>0$ arbitrary and with $T\sim m^\omega$.

  Showing that such an inequality is sufficient to conclude is a simple exercise that we will carry at the end of the proof.

  Assume for simplicity that $T=m^\omega$.

  We assume that the property holds for $\tilde{\epsilon}=\epsilon+\omega \epsilon'$.
  Assume that we can freely replace $D^{X,Y}_{n,m}$ with $\sum_{i,j} D^{i,j}_{n,m}$, which should not be much of a trouble considering the previous proofs in the paper.

  %


  We discuss depending on whether the most important contribution is more or less than  $ T^{-1+\epsilon'} m^{-\frac{1}{2}+\epsilon}$.
  In the first case, for this couple $(i,j)$,
  \[T \big|nmD_{n,m}^{i,j}-\frac{\ell^{i,j}}{4\pi^2}\big|\geq m^{-\frac{1}{2}+\epsilon+\omega \epsilon'} .\]
  For a given couple $(i,j)$, the scaling relation implies that this can occur with probability at most $p_{\epsilon+\omega \epsilon'}$. The probability that this happens for some couple $i,j$ it then at most $T^2 p_{\epsilon+\omega \epsilon'}$, which decays more quickly than any polynomial from our assumption.

  The second case is more subtle. We first remark that the number of couples $(i,j)$ for which $E^{i,j}=\big|nmD_{n,m}^{i,j}-\frac{\ell^{i,j}}{4\pi^2}\big|\neq 0$ is smaller than $T^{1+\epsilon''}$, but on an event with probability that decays more quickly than any power of $T$.

  Since the highest contribution is $ T^{-1+\epsilon'} m^{-\frac{1}{2}+\epsilon}$, the $ T^{1-\epsilon'}/2$ highest contributions sum to less than $ m^{-\frac{1}{2}+\epsilon}$. The other contributions sum to more than $ m^{-\frac{1}{2}+\epsilon}$. Since this is a sum over less than $T^{1+\epsilon''}$ couples, the maximum of them must be larger than $T^{-1-\epsilon''} m^{-\frac{1}{2}+\epsilon}$,
  and each of the $ T^{1-\epsilon'}/2$ highest contributions is therefore also larger than $T^{-1-\epsilon''} m^{-\frac{1}{2}+\epsilon}$. Among these
  $ T^{1-\epsilon'}/2$ couples $(i,j)$, it is extremely unlikely that they all share the same value of $i$ or the same value of $j$ (this is by the same line of reasoning that allowed to control the number of couples for which $E^{i,j}$ is non-zero). In particular, appart from an extremely unlikely event, we can deduce than there exist two couples $(i,j),(k,l)$ with $i\neq k$ and $j\neq l$ and such that $E^{i,j}$ and $E^{k,l}$ are both greater than $T^{-1-\epsilon''} m^{-\frac{1}{2}+\epsilon} $.
  From scaling, for given couples $(i,j),(k,l)$, each of these two events has a probability higher than $p_{\epsilon- \omega \epsilon''}$, and because we were able to assume $i\neq k$ and $j\neq l$, these two events are poorly correlated, so that the probability that these two events occurs at the same time is about   $p_{\epsilon- \omega \epsilon''}^2$ (time a small power of $T$). The probability that such a quadruple $(i,j,k,l)$ exists is therefore about $T^{4} p_{\epsilon- \omega \epsilon''}^2$ at the most, hence the announced inequality.

  %
  %
  %
  %
  %
  %
  %
  %
  %

\end{proof}
The actual proof is quite long, and we split it into several pieces.

\begin{sublemma}
  For a given $j\in \{1,\dots, T\}$ and $\epsilon>0$, let \[T_j=\{i\in \{1,\dots, T\}: |X_{(i-1)T^{-1}}-Y_{(j-1)T^{-1}}|\leq T^{-\frac{1}{2}+\epsilon} \}.\]
  Then, for all $r$, for all $\epsilon>0$, there exists $C$ such that for all $T$ and all $j\in \{1,\dots, T\}$,
  \[ \sup_{x,y} \mathbb{P}_{x,y}( \# T_j \geq T^{3\epsilon}  )\leq C T^{-r}.\]

  In particular,
  \[ \sup_{x,y} \mathbb{P}_{x,y}( \# \{(i,j): |X_{(i-1)T^{-1}}-Y_{(j-1)T^{-1}}|\leq T^{-\frac{1}{2}+\epsilon}  \} \geq T^{1+3\epsilon}  )\leq C T^{-r}.\]
\end{sublemma}
\begin{proof}
  Let $k$ be an integer. Let $E^i$ be the event $|X_{(i-1)T^{-1}}|\leq  T^{-\frac{1}{2}+\epsilon} $ and $T'_i= \{i\in \{1,\dots, T\}: E^i$.

  Then,
  \begin{align*}
  \mathbb{E}_0[ (\# T'_j )^k]
  &\leq   \sum_{ i_1,\dots, i_k \in \{1,\dots, T\}} \mathbb{P}_x( E^{i_1}\cap \dots \cap E^{i_k})\\
  &\leq  \big( \sum_{ \substack{i_1,\dots, i_k \in \{1,\dots, T\}\\ \exists p,q: i_p=i_q }} + k! \sum_{ i_1<\dots < i_k \in \{1,\dots, T\}} \sup_x  \mathbb{P}_x( E^{i_1}\cap \dots \cap E^{i_k})\\
  &\leq   \big( k^2 \mathbb{E}[ (\# T_j)^{k-1}]+k! \sum_{ i_1<\dots < i_k \in \{1,\dots, T\}} \mathbb{P}_0( E^{i_1}\cap \dots \cap E^{i_k})   )
  \end{align*}
  Besides, for $i_1<\dots < i_k$, and setting $i_0=0$,
  \begin{align*}
  \mathbb{P}_0( E^{i_1}\cap \dots \cap E^{i_k})& = \prod_{p=1}^k \mathbb{P}_0( E^{i_p}| E^{i_1},\dots, E^{i_{p-1}})\leq  \prod_{p=1}^k \sup_x\mathbb{P}_0( E^{i_p}|X_{(i_{p-1}-1)T^{-1}}=x)\\
  &\leq  \prod_{p=1}^k \frac{\pi T^{-1+2\epsilon} }{2\pi (i_p-i_{p-1})T^{-1}}\leq 2^{-k} T^{2\epsilon k}\prod_{p=1}^k (i_p-i_{p-1})^{-1}.
  \end{align*}
  The sum over $i_1<\dots < i_k \in \{1,\dots, T\}$ is less than
  \[ C_k \log(T)^k T^{2\epsilon k}.\]

  We get
  \[
  \mathbb{E}_0[ (\# T'_j )^k]
  \leq
  C'_k \big( \mathbb{E}[ (\# T_j)^{k-1}]+\log(T)^k T^{2\epsilon k} \big),\]
  and a direct recursion on $k$ gives
  \[
  \mathbb{E}_0[ (\# T'_j )^k]
  \leq
  C''_k  \log(T)^k T^{2\epsilon k}.\]
  We then remark
  \[
  \sup_{x,y} \mathbb{P}_{x,y}( \# T_j \geq T^{3\epsilon}  )
  \leq
  \sup_{x} \mathbb{P}_{x}( \# T'_j  \geq T^{3\epsilon}  )
  \leq  T^{-3\epsilon k} \sup_x \mathbb{E}_x[ (\# T'_j )^k]\leq C''_k  \log(T)^k T^{-\epsilon k }.\]
  and we conclude by taking $k$ large enough.
\end{proof}

\begin{sublemma}
\label{sub:bridges}
  Let $E^{i,j}$ be an event depending of $X^i,Y^j$, and assume that $E$ is included on
\[\{\|X^i\|_\infty\leq T^{-\frac{1}{2}+\epsilon }\}.\]
  Then,
  \begin{itemize}
  \item For all $i<k$ and $l<j$,
  \[
  \mathbb{P}_{x,y}( E^{i,j}  E^{k,l} )\leq
  \mathbb{P}_{x,y}( E^{i,j} )\sup_{x',y'\in \R^2} \mathbb{P}_{x',y'}( E^{k-i,j-l} ).
  \]
  \item For all $i<k$ and $j>l$,
  \[
  \mathbb{P}_{x,y}( E^{i,j}  E^{k,l} )\leq 2 \Big(1+ \frac{T^{-1+2\epsilon } }{2 (k-i)T^{-1}} \Big) \sup_{x_2}\mathbb{P}_{x_2,y}(  E^{1,l}  ) \sup_{x_1,y'}\mathbb{P}_{x_1,y'}(   E^{i,l}).
  \]
  \end{itemize}
\end{sublemma}
\begin{proof}
  For the first case, we disintegrate with respect to the value of $(X_{iT^{-1}}, Y_{jT^{-1}})$. We obtain
  \begin{align*}
    \mathbb{P}_{x,y}( E^{i,j}  E^{k,l} )
  &= \int_{(\R^2)^2} \hspace{-0.3cm}p_{iT^{-1}}(x,x')p_{jT^{-1}}(y,y')\mathbb{P}_{x,y}( E^{i,j}  E^{k,l} | (X_{iT^{-1}}, Y_{jT^{-1}})=(x',y'))\\
  &\hspace{9.5cm}  \mathbb{P}_{x',y'}( E^{k-i,l-j} )) \d x' \d y'\\
  &\leq \sup_{x',y'\in\R^2} \mathbb{P}_{x',y'}(   E^{k-i,l-j} ))\int_{(\R^2)^2} p_{iT^{-1}}(x,x')p_{jT^{-1}}(y,y')\\
  &\hspace{6.3cm}\mathbb{P}_{x,y}( E^{i,j} | (X_{iT^{-1}}, Y_{jT^{-1}})=(x',y'))  \d x' \d y'\\
  &=\sup_{x',y'\in\R^2} \mathbb{P}_{x',y'}(E^{k-i,l-j})\mathbb{P}_{x,y}( E^{i,j} ).
  \end{align*}

  For the second case, we set $X_{ik}=(X_{(i-1)T^{-1}}, X_{(k-1)T^{-1}})$. Disintegrating with respect to $X_{ik}$ and $Y_{lT^{-1}})$, we obtain
  \begin{align*}
  \mathbb{P}_{x,y}&(E^{i,j} E^{k,l} )
  = \int_{(\R^2)^3} p_{(i-1)T^{-1}}(x,x_1)p_{(k-i)T^{-1}}(x_1,x_2) p_{lT^{-1}}(y,y')\\
  &\hspace{4.4cm}  \mathbb{P}_{x,y'}( E^{i,j-l}    )| X_{ik}=(x_1,x_2) )\mathbb{P}_{x_2,y}(  E^{1,l}  | Y_{lT^{-1}}=y' ) \d x_1\d x_2 \d y'\\
  &\leq \int_{(\R^2)^2 } p_{(i-1)T^{-1}}(x,x_1)p_{(k-i)T^{-1}}(x_1,x_2) \sup_{y'\in \R^2}\mathbb{P}_{x,y'}( E^{i,j-l}  | X_{ik}=(x_1,x_2)  ) \\
  &\hspace{5.5cm}  \Big(\int_{\R^2}   p_{lT^{-1}}(y,y')
  \mathbb{P}_{x_2,y}( E^{1,l}   | Y_{lT^{-1}}=y' ) \d y' \Big) \d x_1 \d x_2\\
  &=\int_{(\R^2)^2 } p_{(i-1)T^{-1}}(x,x_1)p_{(k-i)T^{-1}}(x_1,x_2) \sup_{y'\in \R^2}\mathbb{P}_{x,y'}( E^{i,j-l}  | X_{ik}=(x_1,x_2)  ) \\
  &\hspace{11cm}\mathbb{P}_{x_2,y}( E^{1,l}   )   \d x_1 \d x_2\\
  &\leq \sup_{x_2}\mathbb{P}_{x_2,y}(   E^{1,l}   ) \int_{(\R^2)^2 } p_{(i-1)T^{-1}}(x,x_1)p_{(k-i)T^{-1}}(x_1,x_2) \\
  &\hspace{7cm}\sup_{y'\in \R^2}\mathbb{P}_{x_1,y'}( E^{1,j-l}  | X_{(k-i)T^{-1}} =x_2  )\d x_1 \d x_2.
  \end{align*}
  We now have to bound $\mathbb{P}_{x_1,y'}( E^{1,j-l} | X_{(k-i)T^{-1}} =x_2  )$for which we use again the expression of the Radon-Nikodym derivative of the Brownian bridge with respect to Brownian motion. For $E^{1,j-l} )$ included on the event
  \[ \{ |X_{iT^{-1}}-X_{(i-1)T^{-1}}|\leq T^{-\frac{1}{2}+\epsilon }\},\]
  we get
  \begin{align*}
  \mathbb{P}_{x_1,y'}( &E^{1,j-l} )| X_{(k-i)T^{-1}} =x_2  )\leq  \sup_{r:|r|\leq T^{-\frac{1}{2}+\epsilon }}\frac{p_{(k-i-1)T^{-1}}(x_1+r,x_2 )}{p_{(k-i)T^{-1}}(x_1,x_2)   }  \mathbb{P}_{x_1,y'}( E^{1,j-l} ) )
  \\
  &\leq \frac{k-i}{k-i-1} \sup_{r:|r|\leq T^{-\frac{1}{2}+\epsilon }}\frac{p_{(k-i)T^{-1}}(x_1+r,x_2 )}{p_{(k-i)T^{-1}}(x_1,x_2)   }  \mathbb{P}_{x_1,y'}( E^{1,j-l} ) )\\
  &= \frac{k-i}{k-i-1} \frac{p_{(k-i)T^{-1}}(0, \max (0,1-\tfrac{T^{-\frac{1}{2}+\epsilon }}{|x_2-x_1|} )(x_2-x_1)  )}{p_{(k-i)T^{-1}}(x_1,x_2)   }  \mathbb{P}_{x_1,y'}( E^{1,j-l} ) )\\
  &\leq 2 \frac{p_{(k-i)T^{-1}}(0, \max (0,1-\tfrac{T^{-\frac{1}{2}+\epsilon }}{|x_2-x_1|} )(x_2-x_1)  )}{p_{(k-i)T^{-1}}(x_1,x_2)   }  \mathbb{P}_{x_1,y'}( E^{1,j-l} ) ),
  \end{align*}
  so that
  \begin{align*}
  \mathbb{P}_{x,y}&(E^{i,j} E^{k,l}   )\leq
  2\sup_{x_2}\mathbb{P}_{x_2,y}( E^{1,l}   )\\
  &\int_{(\R^2)^2 }p_{(k-i)T^{-1}}(0, \max (0,1-\tfrac{T^{-\frac{1}{2}+\epsilon } }{|x_2-x_1|} )(x_2-x_1)  ) \sup_{y'\in \R^2}\mathbb{P}_{x_1,y'}( E^{1,j-l} ) ) \d x_1 \d x_2.
  \end{align*}
  The integral over $x_2$ is less than $1+ \frac{\pi T^{-1+2\epsilon }}{2\pi (k-i)T^{-1}}$, and we get
  \begin{align*}
  \mathbb{P}_{x_1,y'}( E^{1,j-l} )| X_{(k-i)T^{-1}} =x_2  )\leq
  2 \Big(1+ \frac{T^{-1+2\epsilon }}{2 (k-i)T^{-1}} \Big) \sup_{x_2}\mathbb{P}_{x_2,y}(  E^{1,l}  ) \sup_{x_1,y'}\mathbb{P}_{x_1,y'}(E^{1,l}  ).
  \end{align*}
\end{proof}
Remark that we have missed a lot of factors $T$ in the lemma, but this will not be a trouble for our purpose.

\begin{proof}[Proof of Lemma \ref{le:concentration}]
  \diams Let us first consider the case $\epsilon>\frac{1}{2}$. Then, we can bound the left hand side by
  \[ \sup_{x,y}(\mathbb{P}_{x,y}\big(nmD_{n,m}^{X,Y} \geq \frac{1}{2} m^{-\frac{1}{2}+\epsilon} \big)+
  \mathbb{P}_{x,y}\big(\frac{\ell^{X,Y}(\R^2)}{4\pi^2} \geq \frac{1}{2} m^{-\frac{1}{2}+\epsilon} \big),\]
  and therefore by
  \[ 2^p m^{\frac{p}{2}-\epsilon p}\sup_{x,y}( \mathbb{E}_{x,y}\big[ (nmD_{n,m}^{X,Y})^p \Big] +(4\pi^2)^{-p} \mathbb{E}_{x,y}\big[ (\ell^{X,Y}(\R^2))^p \big])
  \leq  C m^{\frac{p}{2}-\epsilon p} ( \log(n)^c +1 ).\]
  Since $\epsilon>\frac{p}{2}$, it suffices then to choose $p$ sufficiently large, so that $\frac{p}{2}-\epsilon p<-r$.

  \diams Let us now consider the case of genral $\epsilon$ but $r<2\epsilon$ small. For all $\epsilon'>0$, there exists $C$ such that for all $T$, the left hand side in the lemma is bounded by
  \[
  \sup_{x,y} m^{1-2\epsilon} \mathbb{E}_{x,y}\big[  \big(nmD_{n,m}^{X,Y}-\frac{\ell^{X,Y}(\R^2)}{{4\pi^2}} \big)^2 \big]\leq C m^{-2\epsilon+\epsilon'}.
  \]

  \diams For the general case, we assume that $T \sim m^\omega$, for some $\omega$ fixed $\omega$. We
  set
  \[p^+_{n,\epsilon,C}=\sup_{x,y} \mathbb{P}_{x,y}(  nm D^{X,Y}_{n,m}-\frac{\ell^{X,Y}(\R^2)}{4\pi^2} \geq C m^{-\frac{1}{2}+\epsilon } ).
  \]
  Let also $E^{i,j}=n^-m^- D^{i,j}_{n^-,m^-}-\frac{\ell^{i,j}(\R^2)}{4\pi^2}$, $R$ the variable that appears in Proposition \ref{prop:encadrement},
  $T_i= \# \{ j: E^{i,j}\neq 0\}$, $S_j=\# \{i: E^{i,j}\neq 0\}$.Then, we have
  \begin{align*}
  &p^+_{n,\epsilon,C}
  \leq\sup_{x,y} \mathbb{P}_{x,y} (  nm R\geq  \frac{C}{m^{-\frac{1}{2}+\epsilon }}  )
  + \mathbb{P}_{x,y}(\exists i:  T_i \geq T^{\epsilon'} )
  + \mathbb{P}_{x,y}(\exists j: S_j\geq T^{\epsilon'} )
  \\&+ \mathbb{P}_{x,y}(\exists (i,j): E_{i,j}\neq 0, |X_{(i-1)T^{-1}}-Y_{(j-1)T^{-1}}|\geq T^{-\frac{1}{2}+\epsilon'' })
  + \mathbb{P}_{x,y}\big(\exists (i,j): E^{i,j}\geq \frac{C  T^{-\frac{1}{2}}m^{-\frac{1}{2}+\epsilon } }{4} )
  \\
  &+\mathbb{P}_{x,y}( \#  \{(i,j): E^{i,j}\geq \frac{C  m^{-\frac{1}{2}+\epsilon}}{4} T^{-1-\epsilon'} \}\geq T^{\frac{1}{2}},
  \forall i,j T_i\leq T^{\epsilon'}, S_j\leq T^{\epsilon'}, \\
  &\hspace{6.5cm}(E_{i,j}\neq 0 \implies  |X_{(i-1)T^{-1}}-Y_{(j-1)T^{-1}}|\leq T^{-\frac{1}{2}+\epsilon'' }) \big).
  \end{align*}
  We already know that the first four terms decay more quickly than any power of $T$, provided that $\omega< \frac{\epsilon}{3}$ (for the first one, it follows from Proposition \ref{prop:encadrement}).

  $\sup_{x,y}\mathbb{P}_{x,y}(  R\geq  \frac{C}{m^{-\frac{1}{2}+\epsilon }}  )$, $\sup_{x,y}\mathbb{P}_{x,y}(\exists i:  T_i \geq T^{\epsilon'} )$ and $ \mathbb{P}_{x,y}(\exists (i,j) E_{i,j}\neq 0, |X_{(i-1)T^{-1}}-Y_{(j-1)T^{-1}}|\geq T^{-\frac{1}{2}+\epsilon'' })$
  Besides, from the usual scaling argument,
  \[
  \sup_{x,y}\mathbb{P}_{x,y}(\exists (i,j): E^{i,j}\geq \frac{C T^{-\frac{1}{2}} m^{-\frac{1}{2}+\epsilon } }{4} )\leq T^2 p^+_{n,\epsilon+\frac{\omega}{2},\frac{C}{4}},
  \]
  and provided $\epsilon'<\frac{1}{2}$,
  \begin{align*}
  &\mathbb{P}_{x,y}( \#  \{(i,j): E^{i,j}\geq \frac{C m^{-\frac{1}{2}+\epsilon}}{4} T^{-1-\epsilon'} \}\geq T^{\frac{1}{2}},
  \forall i,j T_i\leq T^{\epsilon'}, S_j\leq T^{\epsilon'}, \\
  &\hspace{6.5cm}(E_{i,j}\neq 0 \implies  |X_{(i-1)T^{-1}}-Y_{(j-1)T^{-1}}|\leq T^{-\frac{1}{2}+\epsilon'' }) )\\
  &\leq
  \mathbb{P}_{x,y}(\exists (i,j),(k,l): i\neq k, j\neq l, E^{i,j}\geq \frac{C m^{-\frac{1}{2}+\epsilon}}{4} T^{-1-\epsilon'}, E^{k,l}\geq \frac{m^{-\frac{1}{2}+\epsilon}}{4} T^{-1-\epsilon'}, \\&\hspace{6cm} |X_{(i-1)T^{-1}}-Y_{(j-1)T^{-1}}|\geq T^{-\frac{1}{2}+\epsilon'' })    )\\
  &\leq \sum_{\substack{i,j,k,l=1 \\ i\neq k , j\neq l }}^{T}\mathbb{P}_{x,y}(E^{i,j}\geq \frac{C m^{-\frac{1}{2}+\epsilon}}{4} T^{-1-\epsilon'}, E^{k,l}\geq \frac{m^{-\frac{1}{2}+\epsilon}}{4} T^{-1-\epsilon'} , \\
  &\hspace{4cm}|X_{(i-1)T^{-1}}-Y_{(j-1)T^{-1}}|\leq T^{-\frac{1}{2}+\epsilon'' } , |X_{(k-1)T^{-1}}-Y_{(l-1)T^{-1}}|\leq T^{-\frac{1}{2}+\epsilon'' }   )\\
  &\leq\sum_{\substack{i,j,k,l=1 \\ i\neq k , j\neq l }}^{T} 2\Big( 1+ \frac{T^{-1+2\epsilon''}}{2(k-i)T^{-1}} \Big)\sup_{x,y}\mathbb{P}_{x,y}(E^{i,j}\geq \frac{C m^{-\frac{1}{2}+\epsilon}}{4} T^{-1-\epsilon'})
  \sup_{x,y}\mathbb{P}_{x,y}(E^{k,l}\geq \frac{C m^{-\frac{1}{2}+\epsilon}}{4} T^{-1-\epsilon'})
  \shortintertext{(using Sublemma \ref{sub:bridges})}\\
  &\leq C' \log(T) T^{3+ 2\epsilon''} (p^+_{n, \epsilon-\epsilon'\omega, \frac{C}{4}})^2 \leq C'' T^4 (p_{n, \epsilon-\epsilon'\omega, \frac{C}{4}}^+)^2.
  \end{align*}

  Finally, for $q$ arbitrary
  \[p^+_{n, \epsilon,C}\leq C'T^{-q}+T^2 p^+_{n,\epsilon+\frac{\omega}{2}, \frac{C}{4} }+
  C'T^4 (p^+_{n, \epsilon-\epsilon'\omega, \frac{C}{4}})^2.\]
  Setting
  \[
  r(\epsilon)=\sup\Big\{ r: \exists C: \forall n,m, \sup_{x,y} \mathbb{P}_{x,y}(nmD^{x,y}_{n,m}-\frac{\ell^{X,Y}(\R^2)}\geq C' m^{-\frac{1}{2}+\epsilon } ) \leq C m^{-r} \Big\},
  \]
  which does not depend on $C'$, we get
  \[
  r(\epsilon)\leq  \min( \omega q, -2 \omega+ r(\epsilon+\frac{\omega}{2} ), -4\omega+2 r(\epsilon-\epsilon' \omega).
  \]
  Since $q$ is arbitrary,
  \[
  r(\epsilon)\leq  \min( -2 \omega+ r(\epsilon+\frac{\omega}{2} ), -4\omega+2 r(\epsilon-\epsilon' \omega).
  \]
  Set $\epsilon_0=\inf\{ \epsilon: r(\epsilon)=\infty$, and assume $\epsilon_0>0$. Let then $\omega=\frac{1}{7}r(\epsilon_0/2)$. Remark that we
  know $\epsilon_0<\infty$ (and actually, $\epsilon_0\leq \frac{1}{2}$) from the first case we treated, and we know $\omega>0$ from the second case we treated.

  Then, for all $\epsilon\in \big(\max(\frac{\epsilon_0}{2}, \epsilon_0-\frac{\omega}{2}), \epsilon_0\big)$ which is a continuity point of $r$, taking $\epsilon''$ such that $r(\epsilon-\epsilon''\omega)>\frac{6}{7} r(x)$,
  we have $2r(\epsilon-\epsilon''\omega)-4\omega> \frac{8}{7} r(x)$, so that
  $r(\epsilon)\geq r(\epsilon+ \omega \epsilon')-2\omega$. Yet, $\epsilon+ \omega \epsilon'>\epsilon_0$ so that $r(\epsilon)=\infty$. This contradicts the minimality of $\epsilon_0$. We can conclude that $r$ has no continuity point on $ \big(\max(\frac{\epsilon_0}{2}, \epsilon_0-\frac{\omega}{2}), \epsilon_0\big)$, which is absurd since $r$ is monotonic. We can conclude that $\epsilon_0=0$.

  The probability that $\frac{\ell^{X,Y}(\R^2)}{4\pi^2} -nmD^{X,Y}_{n,m}\geq m^{-\frac{1}{2}+\epsilon}$ is controlled in an identical way, and that concludes the proof.
%
%
%
%
%
%
%
\end{proof}

\begin{corollary}
  For all $p\in [1,\infty)$, for all $r, c,\epsilon>0$,
  \[
  \sup_{x,y} \mathbb{E}_{x,y}\big[\big|m^{-\frac{1}{2}+\epsilon}|(m D^{X,Y}_{n,m}-\ell^{X,Y}_{n,m}(\R^2))\big|^p ]\underset{n^{c}<m<n}= O(m^{-r}).
  \]
\end{corollary}

The almost sure asymptotic then follows directly from Lemma 4.1 in \cite{LAWA}:
\begin{lemma}
\label{le:max}
Let $(D_N)_{N\in \mathbb{N}}$ be a random sequence which is almost surely decreasing and takes non-negative values. Assume that there exists $C\geq 0$, $r\in (0,p)$ and $p>1$ such that, for all $N$ large enough,
\[ \mathbb{E}[ |ND_N- C|^p ]\leq N^{-r }. \] 
Then, for $q<\frac{p-1}{p}r$,
\[ \mathbb{E}\big[\sup_{N\geq N_0} N^{q} |ND_N- C|^p \big]\underset{N_0\to \infty}\longrightarrow 0. \]
\end{lemma}
Remark that, in \cite{LAWA}, $C$ was thought of as deterministic, but neither the lemma nor the proof use such a condition.

Recall that $m=m(n)$, and that we have fix some $c_1,c_2$ such that $n^{c_1}\leq m\leq n^{c_2}$. We apply Lemma \ref{le:max} to the sequence $D_N=D^{X,Y}_{n,m}$, where $n$ is the largest integer such that $nm\leq N$, with $p$ arbitrary large and $r<\frac{c_1}{2(1+c_1)}$. Then, $m^{-\frac{1} {2}}\leq (nm)^{-r}$,
so that we can apply the lemma indeed. Th exact bound given in Theorem \ref{th:intersection} is then obtain by noticing that $(nm)^{-\frac{c_2}{2(1+c_2)}}\leq m^{\frac{1}{2} \frac{c_1(1+c_2)}{c_2(1+c_1)}}$.

\section{Convergence for the joint winding measure}
\label{sec:mesure}
The space of finite measures on $\R^2$ is endowed
with the $1$-Wasserstein distance
\[ W_1(\mu,\nu)=\sup\{ \int_{\R^2} f \d (\mu-\nu): f \mbox{ $1$-Lipschitz}, f(0)=0 \}.\]
Remark that the condition $f(0)=0$ (or a similar normalisation condition) is necessary to deal with non-probability measures.
\subsection{Convergence in $L^2$}
Our goal in this section is to show that for all $p\geq 1$, the normalised joint winding measure $\mu_{n,m}=nm\mathbbm{1}_{\mathcal{D}_{n,m}}\d z$ converges in $L^p$, for the $1$-Wasserstein distance, toward the intersection measure $\ell^{X,Y}$.

\begin{proof}[Proof of Theorem \ref{th:mesure}]
  Let $\mu_{n,m}^{i,j}=nm\mathbbm{1}_{\mathcal{D}^{i,j}_{n,m}}\d z $ and let $\ell_{i,j}$ be te intersection measure of $X_i$ and~$Y_j$. Let also $n^\pm=n+\pm \sqrt{n}(T+1)$, $m^\pm=n\pm \sqrt{n}(T+1)$, and $f^{\pm}$ be the positive (resp. negative) part of $f$.
  \begin{align*}
  \mathbb{E}\Big[ \Big| \int_{\R^2} f^{\pm} \d\mu_{n,m}- \int_{\R^2} f^{\pm} \d{\ell}_{X,Y}   \Big|^p\Big]
&  \leq C_p \Big( \mathbb{E}\Big[ \Big| \int_{\R^2} f^{\pm} \d\mu_{n,m}- \sum_{i,j=1}^{T} \int_{\R^2} f^{\pm} \d\mu_{n^{\pm},m^{\pm}}^{i,j}   \Big|^p\Big] \\
  &+ \mathbb{E}\Big[ \Big( \sum_{i,j=1}^T \Big|\int_{\R^2} f^{\pm} \d\mu^{i,j}_{n^{\pm},m^{\pm}}- f^{\pm}(X_{iT^{-1}}) nm D^{i,j}_{n^{\pm},m^{\pm}} \Big) \Big|^p\Big]\\
&  +\mathbb{E}\Big[ \Big|\sum_{i,j=1}^T \Big(f^{\pm}(X_{iT^{-1}}) (nm D^{i,j}_{n^{\pm},m^{\pm}}-\ell_{i,j}(\R^2)) \Big) \Big|^p\Big] \\
  &+\mathbb{E}\Big[ \Big|\sum_{i,j=1}^T\Big( f^{\pm}(X_{iT^{-1}}) \ell_{i,j}(\R^2) -\int f^{\pm}  \d\ell_{i,j}\Big)  \Big|^p\Big]
  \Big).
  \end{align*}

  \diams
  To control the first term, remark first that the supremum of $|f^{\pm}|$ over $B(0,K)$ is smaller than $K$, so that the essential supremum of $|f^{\pm}|$ over the measures we look at is smaller than $\|X\|_\infty$. We deduce that
  \[\Big|\int_{\R^2} f^\pm \d\mu_{n,m}- \sum_{i,j=1}^{T} \int_{\R^2} f^\pm \d\mu_{n^\pm,m^\pm}^{i,j} \Big|\leq nm \|X\|_\infty R,\]
  with $R$ given in Proposition \ref{prop:encadrement}.
  We get
  \begin{align*}
  &\mathbb{E}\Big[ \Big| \int_{\R^2} f^{\pm} \d\mu_{n,m}- \sum_{i,j=1}^{T} \int_{\R^2} f^{\pm} \d\mu_{n^{\pm},m^{\pm}}^{i,j}   \Big|^p\Big]\\
  &\leq n^pm^p \mathbb{E}\big[ \|X\|_\infty^{2} R^p\big]
  \leq n^pm^p \mathbb{E}\big[ \|X\|_\infty^{4}\big]^{\frac{1}{2}} \mathbb{E}\big[R^{2p}\Big]^\frac{1}{2}\\
  &\leq C \log(T)^c T^{3p-2} m^{-p}.
  \end{align*}

  \diams
  We now look at the second term. For $\alpha=\frac{1}{2}-\epsilon$, we have
  \begin{align*}
  \Big|\int_{\R^2} f^{\pm} \d\mu^{i,j}_{n^{\pm},m^{\pm}}- f^{\pm}(X_{iT^{-1}}) nm D^{i,j}_{n^{\pm},m^{\pm}}\Big|
  &=\Big|\int_{\R^2} (f^{\pm}(z)- f^{\pm}(X_{iT^{-1}}))   \d\mu^{i,j}_{n^{\pm},m^{\pm}}(z)\Big|\\
  &\leq \int_{\R^2} |z-X_{iT^{-1}}| \d\mu^{i,j}_{n^{\pm},m^{\pm}}(z)\\
  &\leq \|X\|_{\mathcal{C}^\alpha} T^{-\alpha} \mu^{i,j}_{n^{\pm},m^{\pm}}(\mathbb{R}^2)\\
  &\leq \|X\|_{\mathcal{C}^\alpha} T^{-\alpha} nm D^{i,j}_{n,m},
  \shortintertext{so that}
  \mathbb{E}\!\Big[
  \Big|\!\sum_{i,j}\! \big(\!\int_{\R^2}\! f^{\pm} \d\mu^{i,j}_{n^{\pm},m^{\pm}}\!-\! f^{\pm}(X_{iT^{-1}}) nm &D^{i,j}_{n^{\pm},m^{\pm}}\!\big)\!\Big|^p\Big]
  \leq T^{-p\alpha} \mathbb{E}[\|X\|_{\mathcal{C}^\alpha}^{2p}]^\frac{1}{2} n^pm^p \mathbb{E}\Big[\! \big(\sum_{i,j} D^{i,j}_{n^{\pm},m^{\pm}}\big)^{2p}\Big]^{\frac{1}{2}}\\
  &\leq C T^{-\frac{p}{2}+p\epsilon} n^pm^p \big( \mathbb{E}[(D_{n^{\pm},m^{\pm}}^{X,Y})^{2p}]^{\frac{1}{2}}+ \mathbb{E}\big[R^{2p}\big]^\frac{1}{2}\big)\\
  &\leq C' \log(n)^c T^{-\frac{p}{2}+p\epsilon} \qquad \mbox{(using Lemma \ref{le:globalBound}).}
  \end{align*}

  \diams We now look at the third term. We decomposite it into two parts, depending on $\|X\|_\infty$. First,
  \begin{align*}
  \mathbb{E}\Big[\mathbbm{1}_{\|X\|_\infty\leq T^\epsilon }\Big| \sum_{i,j=1}^T \Big(f^{\pm}(X_{iT^{-1}}) (nm D^{i,j}_{n^{\pm},m^{\pm}}-&\ell_{i,j}(\R^2)) \Big) \Big|^p\Big]
  \leq T^{2p+2\epsilon} \max_{i,j} \mathbb{E}\big[  |nm D^{i,j}_{n^{\pm},m^{\pm}}-\ell_{i,j}(\R^2)|^p\big]\\
  &\leq T^{p+2\epsilon} \sup_{x,y}\mathbb{E}_{x,y}\big[|nm D^{X,Y}_{n^{\pm},m^{\pm}}-\ell^{X,Y}(\R^2)|^p]\\
  &\leq CT^{p+2\epsilon} m^{-p+\epsilon }.
  \end{align*}
Secondly, for $r$ arbitrary large, there exists $C$ such that for all $T$,
\begin{align*}
\mathbb{E}\Big[\mathbbm{1}_{\|X\|_\infty\geq T^\epsilon }&\Big| \sum_{i,j=1}^T \Big(f^{\pm}(X_{iT^{-1}}) (nm D^{i,j}_{n^{\pm},m^{\pm}}-\ell_{i,j}(\R^2)) \Big) \Big|^p\Big]\\
&\leq \mathbb{P}(\mathbbm{1}_{\|X\|_\infty\geq T^\epsilon })^{\frac{1}{2}}\mathbb{E}\Big[\Big| \sum_{i,j=1}^T \Big(f^{\pm}(X_{iT^{-1}}) (nm D^{i,j}_{n^{\pm},m^{\pm}}-\ell_{i,j}(\R^2)) \Big) \Big|^{2p}\Big]^{\frac{1}{2}}\\
&\leq C T^{-r} T^p \sup_{x,y} \mathbb{E}_{x,y}\Big[|nm D^{X,Y}_{n^{\pm},m^{\pm}}-\ell^{X,Y}(\R^2)|^{2p} \Big]^{\frac{1}{2}}\\
&\leq C' T^{-r} T^p  \log(n)^c n^pm^p.
\end{align*}

  \diams
  Finally,
  \begin{align*}
  |f^{\pm}(X_{iT^{-1}}) \ell_{i,j}(\R^2) -\int f^{\pm}  \d\ell_{i,j}|
  &\leq \int |f^{\pm}-f^{\pm}(X_{iT^{-1}})|  \d\ell_{i,j}\\
  &\leq T^{-\alpha }\|X\|_{\mathcal{C}^\alpha} \ell_{i,j}(\R^2),
  \shortintertext{so that}
  \mathbb{E}\Big[ \Big|\sum_{i,j=1}^T\Big( f^{\pm}(X_{iT^{-1}}) \ell_{i,j}(\R^2) -\int f^{\pm}  \d\ell_{i,j}\Big)  \Big|^p\Big]
  &\leq T^{-\frac{p}{2}+p\epsilon} \mathbb{E}\Big[ \|X\|_{\mathcal{C}^\alpha}^p \Big(\sum_{i,j=1}^T \ell_{i,j}(\R^2)\Big)^p\Big]\\
  &=T^{-\frac{p}{2}+p\epsilon}\mathbb{E}[ \|X\|_{\mathcal{C}^\alpha}^p \ell^{X,Y}(\R^2)^2]\\
  &\leq T^{-\frac{p}{2}+p\epsilon} \mathbb{E}[ \|X\|_{\mathcal{C}^\alpha}^{2p}]^{\frac{1}{2}} \mathbb{E}[\ell^{X,Y}(\R^2)^{2p}]^{\frac{1}{2}}.
  \end{align*}

  All together, we obtain
  \[\mathbb{E}\Big[ \Big| \int_{\R^2} f^{\pm} \d\mu_{n,m}- \int_{\R^2} f^{\pm} \d{\ell}_{X,Y}   \Big|^p\Big]\leq C(\log(T)^c T^{3p-2} m^{-p}+ T^{-\frac{p}{2}+p\epsilon}+T^{p+2p\epsilon} m^{-p+\epsilon } ).\]
  For $T$ a small enough power of $m$ and $\epsilon$ sufficiently small, this goes to $0$ as $n\to \infty$, which concludes the proof.
\end{proof}

\subsection{Almost sure convergence}
Our goal in this section is to show that almost surely, $\mu_{n,m}$ converges toward $\ell^{X,Y}$, for the $1$-Wasserstein distance --or equivalently, in the weak sense.
\begin{proof}
  We follow the same pattern as in the previous proof. We fix $\alpha<\frac{1}{2}$ and some constants $K,K_\alpha,l,D$, and we work on the intersection of the event $\|X\|_\infty\leq K$, $\|X\|_{\mathcal{C}^\alpha}\leq K_\alpha$, $\ell^{X,Y}\leq l$, $|nm D_{n,m}- \ell^{X,Y}|\leq D$. We assume the function $f$ is positive, the negative part is treated similarily.

  Let us denote $R=R_{n,m}$ the random variable that appears in Proposition \ref{prop:encadrement}. Remark that, for all $\epsilon>0$ and $p>1$,
  \[
  \mathbb{P}( nm R_{n,m}\geq m^{-\frac{1}{2}+\epsilon } )\leq \frac{\mathbb{E}[(nmR_{n,m})^p ]}{m^{-\frac{p}{2}+p\epsilon }}\leq C \log(T)^c T^{3p-2}m^{-p\epsilon},
  \]
  so that this probability decays more quickly than any power of $m$, and therefore more quickly than any power of $n$ (recall once gain that $m\geq n^{\epsilon'}$ for some $\epsilon'$). It follows that
  \[
  \mathbb{P}(\exists n\geq n_0:  nm R_{n,m}\geq m^{-\frac{1}{2}+\epsilon } ).
  \]
  decays more quickly than any power of $n_0$, and in particular tends to $0$ as $n_0\to \infty$. We can therefore fix $n_0$ and work on the event
\[ \forall n\geq n_0,  nm R_{n,m}\leq m^{-\frac{1}{2}+\epsilon }.\]
Finally, for $\epsilon>0$
\[
\mathbb{P}(\exists i,j\in \{,\dots, t\}: |nm D^{i,j}_{n,m} -\ell^{i,j}_{n,m}|\geq T^{-1} m^{-\frac{1}{2}+\epsilon} )
\leq T^2 \sup_{x,y} \mathbb{P}_{x,y}(   |nm D^{X,Y}_{n,m} -\ell^{X,Y}_{n,m}|\geq  m^{-\frac{1}{2}+\epsilon} ).
\]
From Lemma \ref{le:concentration}, we know the right-hand side decays to $0$ more quickly than any power of $m$.  We can therefore fix $n_1$ and work on the event
\[ \forall n\geq n_1, \forall i,j\in \{1,\dots, T\},|nm D^{i,j}_{n,m} -\ell^{i,j}_{n,m}|\geq T^{-1} m^{-\frac{1}{2}+\epsilon} .\]
In the following, we assume $n\geq n_0$ and $n\geq n_1$.

  \diams We deduce  \[
  \Big| \int_{\R^2} f \d\mu_{n,m}- \sum_{i,j=1}^{T} \int_{\R^2} f \d\mu_{n,m}^{i,j}   \Big|
\leq nm \|X\|_\infty R_{n,m}\leq K m^{-\frac{1}{2}+\epsilon }\underset{n\to \infty}\longrightarrow 0.\]

  \diams Besides,
  \[
  \Big|\int_{\R^2} f \d\mu^{i,j}_{n,m}- f(X_{iT^{-1}}) nm D^{i,j}_{n,m}\Big|
  \leq \|X\|_{\mathcal{C}^\alpha} T^{-\alpha} nm D^{i,j}_{n,m},
  \]
  so that
  \begin{align*}
  \Big|\sum_{i,j=1}^T\int_{\R^2} f\d\mu^{i,j}_{n,m}- f(X_{iT^{-1}}) nm D^{i,j}_{n,m}\Big|
  &\leq \|X\|_{\mathcal{C}^\alpha} T^{-\alpha} nm \sum_{i,j=1}^T D^{i,j}_{n,m}
  \leq \|X\|_{\mathcal{C}^\alpha} T^{-\alpha} nm (D_{n^-,m^-} +R)\\
  &\leq K_\alpha (D+1) T^{-\alpha}\underset{n\to \infty}\longrightarrow0.
  \end{align*}

  \diams Besides,
  \[ \Big| \sum_{i,j=1}^T \Big(f(X_{iT^{-1}}) (nm D^{i,j}_{n,m}-\ell_{i,j}(\R^2)) \Big) \Big|
  \leq K \sum_{i,j=1}^T |nm D^{i,j}_{n,m}-\ell_{i,j}(\R^2)|\leq K T  m^{-\frac{1}{2}+\epsilon}\underset{n\to \infty}\longrightarrow0.
  \]

  %
%
  \diams
  Finally,
\[ \sum_{i,j=1}^T|f^{\pm}(X_{iT^{-1}}) \ell_{i,j}(\R^2) -\int f^{\pm}  \d\ell_{i,j}|\leq
\sum_{i,j=1}^T T^{-\alpha }\|X\|_{\mathcal{C}^\alpha} \ell^{i,j}(\R^2)=
T^{-\alpha }\|X\|_{\mathcal{C}^\alpha} \ell^{X,Y}(\R^2)\underset{n\to \infty}\longrightarrow 0,\]
  which concludes the proof.
\end{proof}

\section{Acknowledgements}
I am deeply thankful to Thierry Lévy for the help he brought to me,
and to Nathana\"el Berestycki,  Quentin Berger,  Christophe Garban, Jean-François Le Gall, Perla Sousi, and Zhan Shi for their questions and comments before and during the defence of my PhD. They helped me a lot in the process of giving shape to the frivolous ideas I had in mind.

\section{Funding}
I am pleased to acknowledge support from the ERC Advanced Grant 740900 (LogCorRM) during the time I was writing this paper.

\bibliographystyle{plain}
\bibliography{bib.bib}

\end{document}